\documentclass[10pt]{article}
\usepackage{amsmath, amsthm, amssymb}

\newtheoremstyle{theorem}
{10pt} 
{10pt} 
{\sl} 
{\parindent} 
{\bf} 
{. } 
{ } 
{} 
\theoremstyle{theorem}

\newtheoremstyle{defi}
{10pt} 
{10pt} 
{\rm} 
{\parindent} 
{\bf} 
{. } 
{ } 
{} 
\theoremstyle{defi}

\newtheorem{df}{Definition}[section]

\newtheorem{thm}[df]{Theorem}

\newtheorem{lem}[df]{Lemma}

\newcommand{\F}{\noindent}

\newcommand{\SP}{\smallskip}
\newcommand{\MP}{\medskip}
\newcommand{\BP}{\bigskip}

\newcommand{\beq}{\begin{eqnarray}}
\newcommand{\ene}{\end{eqnarray}}

\newcommand{\beqs}{\begin{eqnarray*}}
\newcommand{\enes}{\end{eqnarray*}}

\newcommand{\eq}[1]{(\ref{#1})}

\newcommand{\nom}{\nonumber}

\newcommand{\R}{{{\mathbb R}}}
\newcommand{\N}{{\mathbb N}}

\newcommand{\Q}{{\mathbb Q}}

\newcommand{\eee}{\mbox{\boldmath $\bar e$}}
\newcommand{\aaa}{\mbox{\boldmath $\bar a$}}
\newcommand{\bbb}{\mbox{\boldmath $\bar b$}}

\newcommand{\kkk}{\mbox{\boldmath $\bar k$}}
\newcommand{\nnn}{\mbox{\boldmath $\bar n$}}
\newcommand{\ppp}{\mbox{\boldmath $\bar p$}}
\newcommand{\qqqq}{\mbox{\boldmath $\bar q$}}
\newcommand{\ooo}{\mbox{\boldmath $\bar 0$}}
\newcommand{\llll}{\mbox{\boldmath $\bar 1$}}
\newcommand{\tttt}{\mbox{\boldmath $\bar 2$}}

\newcommand{\pmat}{\begin{pmatrix}}
\newcommand{\emat}{\end{pmatrix}}

\newcommand{\DD}{{\cal{D}}}

\newcommand{\EEE}{{\bf E}}
\newcommand{\GGG}{{\bf G}}
\newcommand{\HHH}{{\bf H}}
\newcommand{\RRR}{{\bf R}}
\newcommand{\SSSS}{{\bf S}}
\newcommand{\TTTT}{{\bf T}}

\newfont{\blg}{cmr10 scaled \magstep3}
\newfont{\bg}{cmr10 scaled \magstep4}
\newfont{\bgg}{cmr10 scaled \magstep5}

\newcommand{\bigzerou}{\smash{\lower1.7ex\hbox{\bg 0}}}

\newcommand{\bigzerouu}{\smash{\lower1.7ex\hbox{\bgg 0}}}

\newcommand{\bigzerow}{\smash{\lower1.0ex\hbox{\blg 0}}}
\newcommand{\bigzeroww}{\smash{\lower0.4ex\hbox{\blg 0}}}
\newcommand{\sharpl}{\smash{\lower1.7ex\hbox{$\sharp$}}}
\newcommand{\sharplp}{\smash{\lower1.7ex\hbox{$\sharp'$}}}

\newcommand{\xxxa}{\mbox{$\lceil x_1\rceil$}}
\newcommand{\xxxn}{\mbox{$\lceil x_n\rceil$}}

\renewcommand{\eee}{\mbox{$\lceil e\rceil$}}
\renewcommand{\aaa}{\mbox{$\lceil a\rceil$}}
\renewcommand{\bbb}{\mbox{$\lceil b\rceil$}}

\renewcommand{\kkk}{\mbox{$\lceil k\rceil$}}
\renewcommand{\nnn}{\mbox{$\lceil n\rceil$}}
\renewcommand{\ppp}{\mbox{$\lceil p\rceil$}}
\renewcommand{\qqqq}{\mbox{$\lceil q\rceil$}}
\renewcommand{\ooo}{\mbox{$\lceil 0\rceil$}}
\renewcommand{\llll}{\mbox{$\lceil 1\rceil$}}
\renewcommand{\tttt}{\mbox{$\lceil 2\rceil$}}

\newcommand{\qqqqo}{\mbox{$\lceil q^{(0)}\rceil$}}
\newcommand{\qqqqa}{\mbox{$\lceil q^{(1)}\rceil$}}
\newcommand{\qqqqn}{\mbox{$\lceil q^{(n)}\rceil$}}
\newcommand{\qqqqomega}{\mbox{$\lceil q^{(\omega)}\rceil$}}
\newcommand{\qqqqomegaa}{\mbox{$\lceil q^{(\omega_1)}\rceil$}}

\begin{document}

\title{An implication of G\"odel's incompleteness theorem}
\author{Hitoshi Kitada\\
Graduate School of Mathematical Sciences\\
University of Tokyo\\
Komaba, Meguro, Tokyo 153-8914, Japan\\
kitada@ms.u-tokyo.ac.jp}
\date{March 22, 2009}
\maketitle

\begin{abstract}
A proof of G\"odel's incompleteness theorem is given. With this new proof a transfinite extension of G\"odel's theorem is considered. It is shown that if one assumes the set theory ZFC on the meta level as well as on the object level, a contradiction arises. The cause is shown to be the implicit identification of the meta level and the object level hidden behind the G\"odel numbering. An implication of these considerations is stated.

{\bf AMS Subject Classification:} 03F40, 03F15, 03B25, 03E99

{\bf Key Words and Phrases:} Consistency, Incompleteness, G\"odel, Transfinite extension
\end{abstract}

\section{Introduction}\label{chap:1}

G\"odel's incompleteness theorem is well-known. It says ``If a formal theory $S$ including number theory is consistent, there is a proposition $G$ both of whose affirmation $G$ and negation $\neg G$ are not provable in $S$." This is the G\"odel's first theorem. The second theorem says ``If a formal theory $S$ including number theory is consistent, the consistency of $S$ is not provable by the method which is formalizable in the theory $S$ itself."

The incompleteness in this context is the incompleteness in syntactic sense, i.e. it is regarded to be proved solely by the syntactic treatment of the words of the formal system $S$, independently of the meaning of the propositions. There is on the other hand the completeness in semantic sense. According to this definition, a completeness in a restricted sense is known to hold for a subsystem of $S$.

The present paper is written in response to the invitation for publication in International Journal of Pure and Applied Mathematics from the Editor-in-Chief Professor Drumi Bainov and Managing Editor Professor Svetoslav Nenov of the journal, and will introduce the author's recent result in foundations of mathematics which appeared in Mathematics for Scientists vol. 41, No. 4 - vol. 42, No. 3, Gendai-Suugaku-Sha, April, 2008 - March, 2009 \cite{[Kitada-godel]} in Japanese language. English was not sufficient for his purpose so that he preferred the language in writing the result. However as many people seem not to be familiar with Japanese language, he thought it would be helpful to those people to write an introductory explanation to the result in \cite{[Kitada-godel]}.



\subsection{Incompleteness theorem}\label{1.1}

\normalsize

As stated, syntactic incompleteness is the incompleteness consequent solely due to the formal treatment of the words and is regarded as the one which would hold independently of the meaning of the words.

Namely a theory $S$ has primitive symbols, and formal expressions which are sequences of those primitive symbols. We call ``words" the expressions that are constructed from primitive symbols by a set of definite rules. Other expressions are discarded as meaningless expressions. Words are divided into terms and formulae (well-formed formula, or in abbreviation, wff), where terms express individual objects and formulae express propositions, theorems, etc. In formulae there are wff's which are not correct although they have meanings. To exclude those incorrect formulae, we choose some obviously correct propositions and set them as axioms from which all of our reasoning starts. Only the propositions derived from those axioms by a set of rules of inference are regarded correct and are called theorems of the theory. The set of all the theorems is identified with the theory $S$. If, for any given meaningful proposition $A$ in the theory $S$, either of the affirmation $A$ or the negation $\neg A$ is a theorem of $S$, every meaningful proposition of $S$ is determined to be true or not. In such a case $S$ is called complete. If otherwise there is a proposition $A$ both of whose affirmation $A$ and negation $\neg A$ are not derivable from the axioms, the theory $S$ has a proposition whose validness is not decided by logical inferences from the axioms. In such a case the theory $S$ is called incomplete.

G\"odel's incompleteness theorem means that if we consider a number theory $S$ and assume that $S$ is consistent, then $S$ is incomplete. This theorem was proved in Kurt G\"odel's paper \cite{[G]}, \"Uber formal unentsceidebare S\"atze der Principia mathematica und verwandter Systeme I, Monatshefte f\"ur Mathematik und Physik, {\bf 38} (1931), 173-198.


\Large

\subsection{Outline of the proof of incompleteness theorem}\label{1.2}


\normalsize

The outline of the proof of the incompleteness theorem is as follows. Our work is to construct a proposition $G$ whose affirmation and negation is not provable in the number theory $S$. Such a proposition is generally the one whose meaning is interpreted as
$$
G=\mbox{``$G$ is not provable."}
$$
Suppose that such a $G$ is constructed in $S$. If we assume that $G$ is provable, then by the meaning of $G$, $G$ would not be provable, a contradiction. If otherwise we assume that the negation $\neg G$ is provable, then by the meaning of negation, $G$ is provable, contradicting $\neg G$. In either case $S$ is inconsistent. However as we have assumed that $S$ is consistent, the conclusion that $S$ is inconsistent is wrong. Thus we have to conclude that both of $G$ and $\neg G$ are not provable.

This is the essential part of the proof of G\"odel's incompleteness theorem.

To perform such an argument rigorously, we define primitive logical symbols, primitive predicate symbols, primitive function symbols, primitive object symbols, variable symbols, parentheses, and comma, and define terms of $S$ as those sequences constructed by finite repetitions of applications of the definite rules to primitive function symbols, primitive object symbols, variables, and parentheses. We then give a set of rules and define formulae of the theory $S$ as those expressions constructed by finite repetitive applications of the rules to terms, primitive logical symbols, primitive predicate symbols, variables, and parentheses.
Such a construction by repetitive applications of a finite number of rules is called a recursive or inductive definition.

In the case of number theory, it suffices to assume the primitive logical symbols like
$$
\Rightarrow,\ \ \wedge,\ \ \lor,\ \ \neg,\ \ \forall,\ \ \exists
$$
It is possible to reduce the number of symbols. However in that case we will need to introduce derived logical symbols to shorten the expressions. To avoid such complexity we assume the usual symbols as above. The meaning of those symbols is from the left to right as follows: ``imply," ``and," ``or," ``not," ``for all," ``there exists."

The primitive predicate symbol of number theory is just the equality
$$
=
$$
and primitive function symbols are as follows:
$$
+,\ \ \cdot,\ \ {}'
$$
where $+$ means summation, $\cdot$ is multiplication and the last prime ${}'$ means the successor. Namely a successor $s'$ of a term $s$ is $s+1$.

The primitive object symbol is just
$$
0
$$
which means zero. Every natural number is thought as some successor $0^{\prime\prime\dots\prime}$ of $0$.

The variables are
$$
a,\ b,\ c,\ \dots,\  x,\ y,\ z,\ \dots
$$
and parentheses are
$$
(\ \ ),\ \ \{\ \ \},\ \ [\ \ ],\ \ \dots
$$
Other auxiliary but important symbol is comma:
$$
,\ 
$$

Terms $s,t,r,\dots$ express natural numbers in the number theory $S$. The actual forms of the terms are the ones like $0^{\prime\prime\dots\prime}$ constructed from object symbol $0$ and function symbol ${}^\prime$, or the successors of variables as $a'$, $b''$, or the ones as $s+t$, $s\cdot t$ constructed from those by summation or multiplication, and so on. Formulae are constructed from terms $s,t,r,\dots$ with using predicate and logical symbols like $s=t+r$, $s=t\cdot r$, $\forall x (x=x)$, $\forall x\exists y(\neg(x=0)\Rightarrow(x=y'))$, $\dots$.

We choose some correct formulae from the set of formulae and assume them as the axioms of the theory $S$:
$$
A_1,\ A_2,\ \dots, \ A_k.
$$
Formulae derived from those axioms by applying rules of inference are called theorems of the theory $S$. As an example of rules of inference, the following is the famous modus ponens:
\begin{quotation}
If formulae $A$ and $A\Rightarrow B$ are true, then the formula $B$ is true.
\end{quotation}
This rule is also called syllogism. Formally this is written:
\begin{eqnarray*}
{\displaystyle{\frac{A,\quad (A)\Rightarrow (B)}{B}}}.
\end{eqnarray*}
Another rule is the rule of inference of predicate calculus
\begin{eqnarray*}
{\displaystyle{\frac{(C)\Rightarrow (F) }{(C)\Rightarrow (  \forall x  \left( F\right)  ) }}},
\end{eqnarray*}
where it is assumed that the formula $C$ does not include the variable $x$.

Applying recursively those rules to the axioms of the theory $S$, the set of all theorems is constructed.


\Large

\subsection{Self-referential proposition}\label{1.3}

\normalsize

The proposition $G$ in the previous subsection whose affirmation and negation are not provable is a proposition which refers to itself. It has been known since the Greek age that some of such self-referential propositions produce contradictions. For instance the well known Cretan paradox is stated in a sharper form as follows
\begin{quotation}
``This sentence is false."
\end{quotation}
The phrase ``this sentence" in the sentence refers to the sentence above itself. So if this sentence is true, then this sentence must be false by the meaning of the sentence, while if this sentence is false, then the negation ``this sentence is true" is true, contradicting the sentence itself, and we have infinite cycles. Or in other way around, either of the assumption ``this sentence is true" or ``this sentence is false" produces the contradiction. Therefore if we assume that our language is consistent, we have to conclude that the truth value of this sentence is not determined.

There are known many such self-referential propositions which produce contradiction. For example, does the sentence
\begin{quotation}
\F
``$n$ is the least natural number which cannot be defined by less than 30 words"\end{quotation}
define a natural number $n$? This sentence consists of 15 words. So if this sentence defines a natural number $n$, then $n$ is defined by less than 30 words, and contradicts the sentence itself. This sentence is also a self-referential sentence.

The following is well-known again:
\begin{quotation}
\F
``The barber of this village shaves all and only those people who do not shave themselves."
\end{quotation}
Does this barber shave himself?

As seen from those, self-reference appears to be the cause of the paradox. Can we then avoid contradictions if we do not make self-reference? In such a case, however, we might not be able to say anything meaningful. In fact recalling our daily conversations, we notice that there are few cases in which we do not make self-reference. If we avoid self-reference, we will have almost nothing to speak.

The G\"odel sentence referred to in the previous subsection
$$
G=\mbox{``$G$ is not provable"}
$$
is also self-referential sentence. However if we look at it more closely, we notice that this sentence refers to the theory $S$ from the level higher than the theory $S$. Those sentences which refer to the object theory $S$ are called the sentences on the meta level. The word ``meta" is a Greek word, meaning ``after" originally. Later it is used also to mean ``higher." The ``meta" in the word ``metamathematics" (which is in some cases used as another name of the research area `foundations of mathematics') is also used in the latter sense.

The G\"odel sentence which speaks about the theory $S$ from a higher level cannot refer to the sentence itself. Because the G\"odel sentence speaks about the object theory $S$ and hence it is on the meta level, so that unless mapping itself into the lower leveled object theory, it cannot speak about itself. In this sense, the actual cause of the contradiction of G\"odel sentence is that we regard the sentences on the meta level as the sentences on the lower object level.

If we think in this way, it would be expected that the G\"odel's incompleteness theorem is not a contradiction that arises merely by the syntactic formal calculus. There is semantic machinery hidden behind the total scenery which we see.


\Large

\subsection{Recursiveness}\label{1.4}

\normalsize

In subsection \ref{1.2} we stated that the expressions for terms and wff are defined recursively. ``Recursive" if mentioned in the context of constructing expressions means generating expressions by applying a finite number of rules repeatedly to definite primitive symbols. Or in the context of constructing theorems, it means generating theorems by repeatedly applying definite rules of inferences to axioms and earlier obtained theorems. An arbitrary introduction of new rules of construction is prohibited in such systems. This sort of restriction is naturally assumed in dealing with daily things by computers or computation machines. It is of course the case that we cannot expect to those automatic machines to deal with newly encountered things by introducing new appropriate rules. This is the same for the rules of organizations or societies. For some term until the rules are corrected, the work to deal with new things is left to the people who encounter those new things and problems. The recursive constructions of terms and wff and the recursive definitions of provability or correctness are mechanical ones and the provability is the concept based only on mechanical operations.

In this sense the G\"odel's incompleteness theorem can be interpreted as meaning that recursive definitions and constructions based on a finite number of groups of axioms and rules of inferences are insufficient. From this standpoint, it can be said that the incompleteness theorem can be taken as giving the possibility of finding new axioms and new rules of inferences.

Actually the axiom of choice and the continuum hypothesis are known to be independent of other axioms of set theory, and hence the theory is consistent even if the negation of those axiom and/or hypothesis is added instead of them. In this way, what is true is dependent on our sense which we see right or correct. G\"odel himself seems to have thought that there is some correct axiom about infinity.

Turning to natural phenomena around us, we notice that there are many phenomena which can be described by recursive methods. For example, if we try to define the fractal figures like ria coastlines, we arrive at recursive definitions. The differential equations which describe weather or waves are often non-linear partial differential equations. Those non-linear partial differential equations have terms in which the solutions appear and hence the solutions influence the solutions themselves recursively and self-referentially.

The fact that the self-referential and recursive description is useful and effective in describing mathematical systems as well as in describing natural phenomena seems to tell that the recursiveness is the fundamental feature of nature not only in humans' introspective considerations like metamathematics.

The description of nature becomes to be useful in making plans for the future only when it is described by mathematics and gives quantitative predictions. Considering those things, we seem to be able to say that it is natural that the foundation of mathematics itself has its basis in recursive definitions, and that the discovery of the recursive method of description has made it possible for humans to describe nature and build plans for the future.


\Large

\subsection{Metamathematics}\label{1.5}


\normalsize

As reviewed, metamathematics or the investigation of foundations of mathematics is a self-referential deed to see mathematics itself in mathematical method. Metamathematics is in this way self-referential and seems to be never related with the outside of mathematics. If we see other mathematics or other areas of science or other human activities, are there any non-self-referential activities? The answer would be left to the reader. However if we recall that the description of nature itself requires the self-referential method, we will see a natural answer.

In either case, we assume, for a moment, as is usually thought that the mathematicians are stubborn or eccentric and naive people who always refer to themselves. The computers invented and worked out by those people, however, come now to be useful and the necessities of daily life. This fact that the introspective mathematics and the metamathematics (or computation sciences), which is the uppermost area among the mathematical sciences, produced the computers useful for daily life would tell that the introspective and inner most activities have become recognized to have their own social values. It is now the age when mathematical predictions and statistical values are used in the political and economic decisions. It is thought that the importance and the value of mathematical thought is recognized by many people now in the age when the introspective ability is understood as an important ability of humans and plays an important role in deciding the future plans of humans.

Metamathematics or foundations of mathematics seemed to have been given almost no attention when the author was young not only in Japan but also in the west and other areas. Recently however it is encouraging and the amenity to see that many young people are making research activities in this field. He thinks that the age will come when people will be able to be engaged more freely and easily in such introspective activities. He hopes this paper would be helpful to those young people.

\section{Formal number theory}\label{chap:2}
\normalsize


As stated at the beginning of the previous section, G\"odel's first incompleteness theorem says ``If a theory $S$ including number theory is consistent, then there is a proposition $G$ which is not provable and refutable\footnote{$G$ is refutable if and only if the negation $\neg G$ is provable.}." A theory $S$ is called complete if for any given proposition $A$ of $S$, one can decide either of $A$ or the negation $\neg A$ is derived by logical inferences from the axioms of the theory $S$. Therefore the incompleteness theorem means that if a theory $S$ is consistent, then it is incomplete. That $S$ is consistent means that for any given proposition $B$, it is not the case that both of $B$ and the negation $\neg B$ are provable. In an inconsistent theory $S$, thus, there is a proposition $B$ such that $B$ and $\neg B$ are both provable, and hence in $S$ every proposition $C$ is provable. The incompleteness theorem above is rephrased as follows: ``A theory $S$ which includes number theory is either inconsistent or incomplete." Further the second incompleteness theorem says ``If a theory $S$ including number theory is consistent, then the consistency of $S$ is not provable by the method formalizable in the theory $S$." The second incompleteness theorem by G\"odel in 1931 is at least on its surface the one which denies the Hilbert formalism's program: ``A mathematical theory is shown to be sound by proving its consistency based on the finitary standpoint." This program was proposed by D. Hilbert to cope with the intuitionism proposed by L. E. J. Brouwer as a criticism to the situation of mathematics which had met several serious difficulties in its foundation around the year 1900. The procedure formalizable in the number theory is thought to be equivalent to the procedure of formal treatment of the words based on the finitary method. Therefore if it is not possible to show the consistency by the method formalizable in the number theory, it would mean that the consistency of number theory is not provable insofar as based on finitary standpoint. This would mean that Hilbert's program is not performable. In this sense, what is essential and important is the problem of consistency and it is not essential whether or not a theory is complete. However the second incompleteness theorem is a corollary of the first incompleteness theorem, so in order to discuss the problem of consistency, it is necessary first to discuss the completeness of the number theory.

\subsection{Formalism}\label{2.1}

\normalsize

It is useful in relation to the later discussion of ours to review here the situation of mathematics around the beginning of the 20th century till the formalism was proposed. As is well-known, at the end of the 19th century, exactly speaking from around the year 1870 to 1900, when the mathematically accurate treatment of real numbers by the use of the concept of set was successfully done by those people like K. Weierstra\ss, R. Dedekind, G. Cantor, there were found several difficulties in the treatment of sets like the paradox about the totality of ordinal numbers found by C. Burali-Forti in 1897, the paradox about the totality of sets found by G. Cantor in 1899, the paradox produced by the totality of the sets each of which does not have itself as an element of itself found by B. Russell in 1902-3, $\dots$. The Russell's paradox was found when he was writing an attempt \cite{Ru} to deduce mathematics from only the axioms of logic. To overcome this difficulty, he introduced the concept of type and order and the axiom of reducibility. In Russell's thought, the primary objects belong to type $0$, the properties about the objects of type $0$ is regarded to belong to type $1$. In the same way the types of $2$, $3$, $\dots$ are defined. Further to consider about the objects like relations or classes, it is necessary to introduce the notion of order\footnote{Actually it was necessary to exclude impredicative definitions. See  the later subsection \ref{12.3}.} inside each class of objects with the same type above type $0$. The introduction of the notion of order makes it impossible to develop the usual analysis\footnote{cf. footnote 2 above and subsection \ref{12.3}.}. To avoid this difficulty Russell introduced the axiom of reducibility\footnote{Chapter II of Vol. I of \cite{W-R}.} that asserts that to every property of the higher order there always corresponds a property of order $0$, which complements the estrangement caused by the classification into different orders. This axiom of reducibility was, as Russell himself later admitted, purely of pragmatic nature\footnote{See the introduction of the second edition of \cite{W-R}, page xiv.} and it is difficult to call it a purely logical axiom. The effort like this to deduce mathematics based only on logic was called logicism, and was given attempts to improve or revise by Carnap, Quine, etc., but there seem to be no followers after then. However the logical analysis of mathematics developed in the monumental work ``Principia Mathematica" by Whitehead and Russell \cite{W-R} gave influence on a theory of continuum and a set theory based on the standpoint of intuitionism (\cite{Hey}) and the construction of formal systems in Hilbert's formalism in the deeper level. Without the effort and contribution by Russell, the modern mathematics might have traced a more winding road. In fact, the title of G\"odel's paper on the incompleteness ``\"Uber formal unentsceidebare S\"atze der Principia mathematica und verwandter Systeme I" tells that the influence of Russell was large. Later H. Weyl \cite{We} criticized that in ``Principia mathematica," ``mathematics is no longer founded on logic, but on a sort of logician's paradise $\dots$." In the present age, logicism is the one forgotten as an attempt to give a basis to mathematics. However its substantial contributions to the modern mathematics should fairly be evaluated.

In the 1880's just when the theory of real numbers based on set theory was successfully done, L. Kronecker gave a criticism that the definitions treated in the theory of real numbers are just ``words," which do not let it possible to determine whether an actual object satisfies them. Later in 1908, L. E. J. Brouwer wrote a paper entitled ``The untrustworthiness of the principles of logic" \cite{B} and developed a criticism that the classical logic which goes back to Aristotle (384-322 B.C.) is derived from the logic applicable to finite sets and hence it is not justified that this logic is applicable to the mathematics of infinite sets. For example, in Euclid's Elements which is thought to be influenced by Aristotle, it is assumed that ``the whole is greater than any proper part." However this does not hold for infinite sets\footnote{Aristotle asserted that the paradox of Zeno of Elea arises from assuming as if the infinity exists. In this sense, the standpoint that the actual infinity does not exist seems to date back to Aristotle.}. Brouwer argued that the problems occur from the unlimited application of ``the law of the excluded middle" to infinite sets. Namely the law states that for any proposition $A$ either of $A$ or its negation $\neg A$ holds. Let for instance the proposition $A$ mean that there exists an element of the set $M$ which satisfies the property $P$. Then the negation $\neg A$ means that every element of $M$ does not satisfy the property $P$. If the set $M$ is a finite set, it is possible to determine whether $A$ or $\neg A$ holds, by checking every element of $M$ one by one. However, if $M$ is an infinite set, it is in principle impossible to perform the check to every element of the infinite set $M$. According to Brouwer, therefore, the law of the excluded middle is the law that should not be applied to general sets including infinite sets. Like this Brouwer's thought is the one based on the finitary standpoint, and is called intuitionism. In his thought, the actual (or existential or completed or extended) infinity is regarded as a fictitious thing. This type of thought goes back to C. F. Gau\ss, 1831 in the modern age, and has something common with the thought of computability in the present age. In fact, Brouwer seems to have thought that ``mathematics is identical with the exact part of our thinking. $\dots$ no science, in particular not philosophy or logic, can be a presupposition for mathematics" (cf. \cite{Hey}). This coincides with the thought of computation which does not assume any philosophy or logic.

The difficulties found at the end of the 19th century required in this way the reinquiry and reexamination into the existing mathematics. These difficulties and the criticism by Brouwer were taken seriously by D. Hilbert and his collaborators P. Bernays, W. Ackerman, J. von Neumann and others. Hilbert proposed the following program. Namely ``classical mathematics which deals with  the infinity is formulated as a formal axiomatic theory, and the treatment of the theory is made based on the finitary standpoint. If the theory is proved to be consistent by this finitary method, it is said that the formal axiomatic theory is sound." In this program, the treatment of the formal theory which deals with the infinity is based on the finitary standpoint which is equivalent to the Brouwer's intuitionism. Therefore if one can perform the program, it would mean that the difficulties and the criticism pointed by Brouwer are avoided. This standpoint is called formalism. In this way, the proof of consistency of a theory is itself considered again in mathematical method. Hilbert called such a mathematical consideration ``metamathematics" or ``proof theory."

As stated in the preface of the present section, the incompleteness theorem that made it impossible to perform the program of formalism is the second incompleteness theorem. However as mentioned, the second incompleteness theorem is a corollary of the first incompleteness theorem except for technical details. Therefore to see the problem in the foundation of mathematics or metamathematics, the primary thing is to show the first incompleteness theorem. In this section, as the first step toward this purpose, we introduce the reader to the notion of formal system and describe how to write down the number theory in the form of a formal system. We state that our description depends somewhat on the descriptions of the books S. C. Kleene, Introduction to Metamathematics, North-Holland Publishing Co. Amsterdam, P. Noordhoff N. V., Groningen, 1964 \cite{K} and H. Kitada and T. Ono, Introduction to Mathematics for Scientists, Gendai-Suugaku-Sha, 2006 \cite{Kitada-book}.


\Large

\subsection{Primitive symbols, terms, and formulae}\label{2.2}

\normalsize

G\"odel's theorem holds for a formal mathematical theory which includes number theory as a subsystem. It is therefore sufficient to prove the theorem for the number theory itself. In this case the theorem reads ``If the number theory $S$ is consistent, there is a proposition $G$ whose affirmation and negation are both unprovable." The point is in the word ``unprovable."

Number theory consists of the ordinal axioms of logic, the rules of inferences and the axioms of number theory. Theorem in the number theory is the proposition which is derivable from the axioms by applying the rules of inferences to them. The incompleteness theorem means that $G$ and the negation $\neg G$ are not obtained by this method.

To prove the incompleteness theorem, it is therefore necessary to write down the axioms of logic and mathematics and the rules of inferences, and need to show that it is not possible to prove the proposition $G$ and the negation $\neg G$ by using those axioms and rules. To grasp the usage of axioms and rules of inferences, it is necessary to determine the primitive symbols and to give the rules to construct propositions by using the symbols. Then it needs to explicitly list the rules of inferences with using those symbols. To do so, we introduce, as in subsection \ref{1.2}, the primitive symbols which are necessary to write down the number theory. Primitive symbols consist of primitive logical symbols, primitive predicate symbols, primitive function symbols, primitive object symbols, variable symbols, parentheses, and comma, as follows:

\begin{enumerate}
\item primitive logical symbols:
\beqs
&&\hskip-20pt\Rightarrow \mbox{ (imply)},\ \ \wedge\mbox{ (and)},\ \ \lor\mbox{ (or)},\ \ \neg\mbox{ (not)},\\
&&\hskip-20pt\forall\mbox{ (for all)},\ \ \exists\mbox{ (there exists)}
\enes
\item primitive predicate symbols: $$= \mbox{ (equals)}$$
\item primitive function symbols:
$$
+\mbox{ (plus)},\ \ \cdot\mbox{ (times)},\ \ {}'\mbox{ (successor (prime))}
$$
\item primitive object symbols: $$0 \mbox{ (zero)}$$
\item variable symbols:
$$
a,\ b,\ c,\ \dots,\  x,\ y,\ z,\ \dots
$$
\item parentheses:
$$
(\ \ ),\ \ \{\ \ \},\ \ [\ \ ],\ \ \dots
$$
\item comma:
$$
,
$$
\end{enumerate}

When $x$ is a variable, the logical expression $\forall x$ is called a universal quantifier and $\exists x$ is called an existential quantifier.

From those symbols, we first define terms which will denote the objects in number theory as follows. This type of definition is called a recursive or inductive definition.

\begin{enumerate}
\item $0$ is a term.
\item A variable is a term.
\item If $s$ is a term, $(s)^{\prime }$ is also a term.
\item If $s, t$ are terms, $(s)+(t)$ is a term.
\item If $s, t$ are terms, $(s) \cdot (t)$ is a term.
\item The only expressions defined by 1-5 are the terms of the number theory.
\end{enumerate}

In particular, the terms in whose construction there does not appear any variable are called numerals or numeral terms.

We next define formula, or well-formed formula (wff) as follows:

\begin{enumerate}
\item If $s$ and $t$ are terms, then $(s)=(t)$ is a formula or wff. The formula of this form is called an atomic formula.
\item If $A, B$ are formulae, then
$$
(A)\Rightarrow (B)
$$
is also a formula.
\item If $A,B$ are formulae,
$$
(A)\wedge (B)
$$
is also a formula.
\item If $A,B$ are formulae,
$$
(A)\lor (B)
$$
is also a formula.
\item If $A$ is a formula,
$$
\neg(A)
$$
is also a formula.
\item If $x$ is a variable and $A$ is a formula, then $\forall x(A) $ is a formula.
\item If $x$ is a variable and $A$ is a formula, $\exists x(A)$ is also a formula.
\item The only expressions defined by 1-7 are the formulae of the number theory.\end{enumerate}


\Large

\subsection{Axioms and rules of inference}\label{2.3}


\normalsize

As stated in the previous section we adopt some of the formulae as the axioms of the number theory, and define theorems or provable formulae as the formulae obtained by applying the rules of inferences to the axioms.

It suffices to assume as the rules of inferences the two rules as mentioned in the previous section. However to simplify the descriptions we assume the following three rules\footnote{Here we follow \cite{K}.}. In the following we assume that the formula $C$ does not contain the variable $x$.

\begin{enumerate}
\item[$I_1$:] Modus ponens. (Syllogism): If the formula $A$ is true and $A\Rightarrow B$ is true, then the formula $B$ is also true.
\begin{eqnarray*}
{\displaystyle{\frac{A,\quad (A)\Rightarrow (B)}{B}}}
\end{eqnarray*}
\item[$I_2$:] Generalization: For any variable $x$, from $F$ follows $\forall x (F)$.
\begin{eqnarray*}
{\displaystyle{\frac{(C)\Rightarrow (F) }{(C)\Rightarrow (\forall x\left(F\right))}}}
\end{eqnarray*}
\item[$I_3$:] Specialization: For any variable $x$, from $F$ follows $\exists x (F)$.
\begin{eqnarray*}
{\displaystyle{\frac{(F)\Rightarrow (C) }{(\exists x\left(F\right))\Rightarrow (C)}}}
\end{eqnarray*}
\end{enumerate}

The axioms of number theory are as follows. In the followings, we omit the unnecessary and obvious parentheses.

The first group consists of the axioms of propositional calculus.

A1. Axioms of propositional calculus. ($A,B,C$ are arbitrary formulae.)
\begin{enumerate}
\item $A\Rightarrow (B\Rightarrow A)$
\item $(A\Rightarrow B)\Rightarrow\left((A\Rightarrow (B\Rightarrow C)) \Rightarrow(A\Rightarrow C) \right)$
\item $A\Rightarrow((A \Rightarrow B)\Rightarrow B)$

(a rule of inference)
\item $A\Rightarrow(B\Rightarrow A\wedge B)$
\item $A\wedge B\Rightarrow A$
\item $A\wedge B\Rightarrow B$
\item $A\Rightarrow A\lor B$
\item $B\Rightarrow A\lor B$
\item $(A\Rightarrow C)\Rightarrow((B\Rightarrow C)\Rightarrow(A\lor B\Rightarrow C))$
\item $(A\Rightarrow B)\Rightarrow((A\Rightarrow \neg B) \Rightarrow \neg A)$
\item $\neg\neg A \Rightarrow A$
\end{enumerate}

The second group consists of the axioms of predicate calculus.

We introduce the following terminologies. If an occurrence of a variable $x$ is in the scope of influence of a quantifier $\forall x$ or $\exists x$, the occurrence is called a bound variable. Otherwise, it is called a free variable.

We call a term $t$ free for $x$ in a formula $A(x)$ which has $x$ as a free variable, if no free occurrence of $x$ in $A(x)$ is in the scope of a quantifier $\forall y$ or $\exists y$ for any variable $y$ of $t$.

A2. Axioms of predicate calculus. ($A$ is an arbitrary formula, $B$ is a formula which does not contain the variable $x$ free, $F(x)$ is a formula which contains a free variable $x$, and the term $t$ is free for $x$ in the formula $F(x)$.)
\begin{enumerate}
\item $(B\Rightarrow A)\Rightarrow
(B\Rightarrow (\forall x A))$ 

(a rule of inference)
\item $\forall x F(x)\Rightarrow F(t)$
\item $F(t)\Rightarrow \exists x F(x)$
\item $(A\Rightarrow B)\Rightarrow ((\exists x A)\Rightarrow B)$

(a rule of inference)
\end{enumerate}

The same rules of inferences appear in the list of axioms to make the same rules of inferences effective inside the formal system of number theory.

The reason that we made an assumption that the term $t$ is free for $x$ is as follows. For instance, let us consider
$$
F(x)=\exists y(x=y)
$$
and let the term $t$ be
$$
t=a+y.
$$
In this case, the term $t$ is not free for $x$ in the formula $F(x)$. If we substitute this term $t$ to the place of the free variable $x$, we obtain
$$
F(t)=\exists y(a+y=y).
$$
The variable $y$ in the term $t$ is bound by the quantifier $\exists y$ and the axiom 2 of the predicate calculus does not hold. Our assumption that the term $t$ is free for $x$ was made to exclude such cases.

The third and fourth groups consist of the axioms of number theory.

A3. Axioms of number theory. ($a,b,c$ are arbitrary variables.)
\begin{enumerate}
\item $a'=b'\Rightarrow a=b$
\item $\neg ( a'=0 )$
\item $a=b\Rightarrow (a=c\Rightarrow b=c)$
\item $a=b\Rightarrow a'=b'$
\item $a+0=a$
\item $a+b'=(a+b)'$
\item $a\cdot 0 =0$
\item $a\cdot b'= a\cdot b  +a $
\end{enumerate}

A4. Axiom of mathematical induction. ($F$ is an arbitrary formula.)
$$
\left(F(0)\wedge\forall x(F(x)\Rightarrow F(x'))\right)\Rightarrow\forall xF(x)
$$


\Large

\subsection{Proof, theorems, and deducibility}\label{2.4}


\normalsize

We make the following definition to define the theorems and the proofs of the formal number theory.

\SP

\begin{df}\label{df2.1}
{\rm 
A formula $C$ is called an immediate consequence of a formula $A$ or two formulae $A, B$ if $C$ is below the line and the other(s) are above the line, in the rules $I_1$, $I_2$ or $I_3$.}
\end{df}

We then define proof, provability and theorem as follows.
\SP

\begin{df}\label{df2.2}
{\rm A finite sequence of formulae, each consecutive pair of which is divided by a comma, is called a formal proof, if each formula $F$ of the sequence is an axiom of number theory or is an immediate consequence of the formula(e) which appear(s) before $F$. A formal proof is said to be the proof of the formula $E$ which appears at the end of the proof, and the formula $E$ is said to be provable in number theory or is called a theorem of number theory.}
\end{df}

If a formula $E$ is deducible in the system in which some formulae are added, $E$ is called deducible from the added assumption formulae.
\SP

\begin{df}\label{df2.3}
{\rm Given a finite number of formulae $D_1,\cdots,D_\ell$ $(\ell\ge 0)$, a finite sequence of formulae is called a formal deduction from the assumption formulae $D_1,\cdots,D_\ell$, if each formula $F$ of the sequence is an axiom or one of the formulae $D_1,\cdots,D_\ell$, or an immediate consequence of the formula(e) which appear(s) before $F$. A deduction is said to be a deduction of its last formula $E$, and the formula $E$ is said to be deducible from the assumption formulae or is called the conclusion of the deduction. We write this as follows.
$$
D_1,\cdots,D_\ell\vdash E.
$$
}
\end{df}

In the case when $\ell=0$, this is written
$$
\vdash E.
$$
This is equivalent with that $E$ is a theorem of number theory.


As we have seen, all definitions of terms, formulae, proofs, theorems, deductions are recursive or inductive definitions. This means that what can be done in a formal system is fundamentally just the mechanical operation. This is related with the concept of computability, which we will not have a chance to touch in this article.

\section{G\"odel numbering}\label{chap:7}

\normalsize


As we have seen in section \ref{chap:2}, the terms, formulae, and theorems are defined by applying a finite number of rules repeatedly to some number of symbols in mechanical way. This procedure of construction is called a recursive or inductive construction as it constructs things by applying the rules of the same form repeatedly.

On the other hand, the procedure which can be described in the formal number theory $S$ is the operation of finite natural numbers, and the mathematical inductions is assumed as an axiom in $S$. Therefore, the recursive procedure of construction of terms, formulae, theorems will be able to be mapped into the operation of natural numbers inside the formal number theory $S$. Namely it will be possible to assign a fixed natural number to each symbol, and from it one can form a rule to assign a unique natural number to each of terms, formulae, proof sequences, etc. If such a rule is made, it will be possible to map the fact that a given sequence of formulae is a proof to a proposition about natural numbers. That a formula $A$ is provable means that there is a proof whose last formula is $A$. Thus it will be possible to express the fact that a given formula $A$ is provable as a proposition in the number theory $S$. As well, it will be possible to express the fact that a given formula $A$ is refutable, i.e. that the negation $\neg A$ of $A$ is provable as a proposition in $S$. A rule that assigns a natural number to each primitive symbol and from this assigns a natural number to a general expression constructed from the primitive symbols in a recursive way is called G\"odel numbering. We denote the natural number which is assigned by this rule to an expression $E$ by $g(E)$, and call it the G\"odel number of the expression $E$. An expression $E$ which has G\"odel number $n$ is expressed as $E_n$. When $E$ is a formula $A$, it is written as $A_n$. Thus $n=g(E_n)$, $n=g(A_n)$, etc. This map $g$ from the totality of expressions to the set $\N=\{0,1,2,\dots\}$ of natural numbers is defined as one to one mapping, but is not onto mapping. Namely $g$ is defined as an injection but is not a surjection in general. Hence for some natural number $m$, there can be the case that there is no expression $E$ such that $g(E)=m$.

There is no problem in assigning natural numbers to symbols and expressions of a formal theory, and thence assigning natural numbers to deductions or proofs. A point that is noted here is that although the assignment which we mentioned looks as if it maps the expressions on the meta level to natural numbers inside the formal system on the object level, the mapping $g$ is actually a map from the set of expressions on the meta level to the set $\N$ of natural numbers on the meta level.

When proving G\"odel's incompleteness theorem, we denote the formula which is obtained by substituting the natural number or numeral $\nnn$ in the formal system:
\beq
\nnn=0^{\overbrace{\prime\prime\dots\prime}^{n\mbox{\scriptsize{ factors}}}}\label{nbar}
\ene
that corresponds to the natural number $n$ on the meta level to the variable $x$ of a formula $F(x)$ by
$$
F(\nnn).
$$
This operation of substitution itself is the one on the meta level. The formula $F({\nnn})$ that is obtained by this substitution is defined by
\begin{eqnarray}
F({\nnn})\stackrel{\scriptsize\mbox{{\it def}}}=\forall x \left(x=n\Rightarrow F\right)\label{dainyuu}.
\end{eqnarray}
Here the reason we do not use the expression $F(x)$ which has $x$ as a variable is that we define $F({\nnn})$ by this formula even when the formula $F$ does not have $x$ as a variable.

We note that in the definition of the formula $F(\nnn)$ which must be a formula in the formal system $S$, there appears the natural number $n$ on the meta level. This fact corresponds to the fact that there appears the natural number $n$ on the meta level to specify the number of primes in the definition of a numeral $\nnn$ of the formal system $S$ in equation \eq{nbar} above (the number $n$ above the $\prime\prime\dots\prime$ on the upper right side of $0$ in the definition \eq{nbar} of $\nnn$). ``Substitution" whatever it looks natural is inevitably a subjective and artificial deed performed by some subject on the meta level. Namely the construction of the numeral $\nnn$ corresponding to $n$ and the substitution of it to $x$ are possible only when some subject recognizes a number $n$ on the meta level.

In fact the formal system $S$ of number theory is regarded as a subsystem of a formal set theory $T$. In this case the number $0$ on the meta level corresponds to the empty set $\emptyset$ in $T$, and $\lceil 0\rceil$ is defined as follows.
$$
\forall x( x=\lceil 0\rceil \Leftrightarrow \forall u(u\not\in x)),
$$
where $A\Leftrightarrow B$ is the abbreviation of $(A\Rightarrow B)\wedge(B\Rightarrow A)$. We define the successor function ${}^\prime$ for a set $m$ by
$$
m'=m\cup \{m\}.
$$
Then the theory of natural numbers is regarded to be a subsystem of the set theory. In the formal number theory $S$ the number $0$ on the meta level is regarded naturally to correspond to the numeral $\lceil 0\rceil$ of the system $S$. However in the case of set theory, there is no such `natural' correspondence, and we have to give an appropriate `correspondence' with considering the original meaning of set theory.


\subsection{G\"odel numbering}\label{7.2}


\normalsize

We are now in a position to give a concrete G\"odel numbering $g$. There are infinitely many ways of giving mappings $g$, and we can take whatever $g$ if it satisfies the properties stated above. We adopt here a variant of the method given in \cite{Kitada-book} of assigning binary numbers to expressions. Namely we first assign natural numbers to primitive symbols as follows.
\beqs
\hskip-22pt&&\begin{array}{cccccccccc}
{}^\prime  &
  0 &
 ( &
 ) &
 \{ &
 \} &
 [ &
 ] &
 + &
 \cdot \\
2^{0}  & 2^{1} & 2^{2}  & 2^{3} & 2^{4} & 2^{5} & 2^{6}  & 2^{7}  & 2^{8} & 2^{9}
\end{array}\\
\hskip-22pt&&\begin{array}{cccccccccc}
 = &
 \Rightarrow &
 \wedge &
 \lor &
 \neg &
 \forall &
 \exists &
 ,  \\
 2^{10} & 2^{11} & 2^{12} & 2^{13} & 2^{14} & 2^{15} & 2^{16} & 2^{17}
\end{array}
\enes
To the expressions constructed from those primitive symbols, we assign G\"odel numbers inductively as follows. First we assign $0$ to the empty expression. I.e. when the G\"odel number is $x=0$, the expression $E_x$ corresponding to $x=0$ is empty, and we regard that there is no expression which corresponds to $0$. Next for two natural numbers $n$, $m$, letting $\ell(m)$ denote the number of figures of the binary expression of $m$, we define the product operation $\star$ by
$$
n\star m= 2^{\ell(m)} \cdot n + m.
$$
Here we define $\ell(m)=0$ for $m=0$.
Now for two expressions $A_1, A_2$ with G\"odel numbers $g(A_1)$, $g(A_2)$, we define the G\"odel number $g(A_1A_2)$ for the connected expression $A_1A_2$ of $A_1$ and $A_2$ in this order by
$$
g(A_1A_2)= g(A_1)\star g(A_2).
$$
This mapping $g$ is obviously one to one.

For example, the G\"odel number of $(0)'$ is calculated as follows. First as the G\"odel number of $($ is $2^{2}$ and that of $0$ is $2^{1}$, we have $n=2^{2}=(100)_2$, $m=2^{1}=(10)_2$ and $\ell(m)=2$. Thus the G\"odel number of $(0$ is
$$
n\star m=2^2\cdot 2^2+2^1=2^4+2^1=(10010)_2.
$$
Namely in binary number the G\"odel number of $($ is $100$, and that of $0$ is $10$. Connecting these consecutively we obtain the G\"odel number $10010$ of $(0$. Similarly the G\"odel number of $(0)$ is $100101000$, and that of $(0)'$ is $1001010001$.

In the definition of G\"odel number above, there is no definition of the G\"odel number of the variables $a,b,c,\dots,x,y$, $z,\dots$. This is because we can express variables by connecting the primitive symbols without overlapping with other expressions like terms or formulae as follows.
\beq
a&\mbox{is}&(0^\prime),\nom\\
b&\mbox{is}&(0^{\prime\prime}),\nom\\
c&\mbox{is}&(0^{\prime\prime\prime}),\nom\\
&&\hskip-18pt\dots\label{hensuu}
\ene
In the following we follow this convention.
\SP

The following lemma will be crucial in the later section \ref{chap:10}.
\SP

\begin{lem}\label{lem7.5}{\sl Let $a\ge 0$ be a natural number and let $w$ be the natural number such that $w'=2^a$, where $w'$ is a successor of the natural number $w$ (i.e. $w'=w+1$). Then we have
\beq
g(a)=2^1\star w\label{2-sin}.
\ene}
\end{lem}
\begin{proof}
The numeral corresponding to a natural number $a\ge0$ is
\vskip-0pt
$$
0^{\overbrace{\prime\prime\dots\prime}^{a\mbox{\scriptsize{ factors}}}}.
$$
\vskip-2pt
By definition
\vskip-2pt
$$
g(0)=2^1,\quad g({}^\prime)=2^0=(1)_2.
$$
\vskip-2pt
Thus
\vskip-2pt
$$
g(a)=g(0^{\overbrace{\prime\prime\dots\prime}^{a\mbox{\scriptsize{ factors}}}})=2^1\star (\overbrace{11\dots 1}^{a\mbox{\scriptsize{ factors}}})_2.
$$
\vskip-2pt
The natural number $w$ such that $w'=2^a$ is in binary expression a sequence of length $a$ of the G\"odel number $2^0=(1)_2$ of prime symbol ${}^\prime$. For instance, if $a=2$, then $w'=w+1=2^a=2^2=(100)_2$ and $w=(11)_2$. Therefore for  $w$ such that $w'=2^a$, we have $w=(\overbrace{11\dots 1}^{a\mbox{\scriptsize{ factors}}})_2$, which proves \eq{2-sin}.
\end{proof}


\Large

\subsection{Incompleteness theorem}\label{7.3}


\normalsize

First we define the following two predicates.



\begin{df}\label{df7.6} {\rm
\begin{enumerate}
\item[ 1)] The predicate $\GGG(a,b)$ means the following.

``A formula $A_a$ with G\"odel number $a$ has just one free variable $x$, and an expression $E_b$ with G\"odel number $b$ is a proof of the formula $A_a=A_a(\aaa)$ obtained from $A_a=A_a(x)$ by substituting $\aaa$ into $x$."
\item[ 2)] The predicate $\HHH(a,b)$ means the following.

``A formula $A_a$ with G\"odel number $a$ has just one free variable $x$, and an expression $E_b$ with G\"odel number $b$ is a proof of the formula $\neg A_a=\neg A_a(\aaa)$ obtained from $\neg A_a=\neg A_a(x)$ by substituting $\aaa$ into $x$."
\end{enumerate}}
\end{df}

We introduce the following notion.

\begin{df}\label{df7.2}{\rm Let $\RRR(x_1,\dots,x_n)$ be a predicate (or relation) about $n(\ge0)$ objects. This predicate is said to be numeralwise expressible in the formal system $S$ if there is a formula $r(u_1,\dots,u_n)$ in $S$ with exactly $n$ free variables $u_1,\dots,u_n$ such that for an arbitrarily given $n$-tuple of natural numbers $x_1,\dots,x_n$, the followings hold.
\begin{enumerate}
\item[ i)] If $\RRR(x_1,\dots,x_n)$ is true, then $\vdash r(\xxxa,\dots,\xxxn)$.\item[ ii)] If $\RRR(x_1,\dots,x_n)$ is false, then $\vdash \neg\hskip1pt r(\xxxa,\dots,\xxxn)$.
\end{enumerate}
In this case, $\RRR(x_1,\dots,x_n)$ is said to be numeralwise expressed by the formula $r(u_1,\dots,u_n)$.
}
\end{df}

It will be shown that the following holds.

\begin{thm}\label{th7.7}{\sl The predicates $\GGG(a,b)$ and $\HHH(a,b)$ in definition \ref{df7.6} are both numeralwise expressed in $S$ by some formulae $g(a,b)$ and $h(a,b)$ respectively.
}
\end{thm}

We now define Rosser formula.

\begin{df}\label{df7.8}{\rm Let $q$ be the G\"odel number of the following formula.
$$
\forall b  \left( 
g(a,b)\Rightarrow \exists c\hskip1pt(c\le b \hskip1pt\land \hskip0pt h(a,c))
\right) .
$$
Namely
\beq
A_{q}(a) = \forall b  \left(g(a,b)\Rightarrow  \exists c\hskip1pt(c\le b \hskip1pt\land \hskip0pt h(a,c))
\right). \nonumber
\ene
Then
\beq
A_{q}(\qqqq)= \forall b  \left(g(\qqqq,b)\Rightarrow \exists c\hskip1pt(c\le b \hskip1pt\land \hskip0pt h(\qqqq,c))
\right). \nonumber
\ene
Here
\beqs
&&\hskip-28ptg(\qqqq,b)=\forall a \left(a=q\Rightarrow g(a,b)\right),\\
&&\hskip-28pth(\qqqq,c)=\forall a \left(a=q\Rightarrow h(a,c)\right).
\enes
$A_q(\qqqq)$ is called Rosser formula.}
\end{df}


\begin{thm}\label{th7.9}{\sl (G\"odel's incompleteness theorem of Rosser type \cite{Rosser}) Let $S$ be consistent. Then neither $A_{q}(\qqqq)$ nor the negation $\neg A_{q}(\qqqq)$ is provable in $S$.}
\end{thm}
\begin{proof} Assume that $S$ is consistent.


Suppose that
\beq
\vdash A_{q}(\qqqq)\mbox{ in } S\label{(5)}
\ene
and let $e$ be the G\"odel number of a proof of $A_{q}(\qqqq)$.
Then by the numeralwise expressibility of $\GGG(a,b)$, we have
\beq
\vdash g(\qqqq,\eee)\label{Rosser1}.
\ene
As we have assumed that $S$ is consistent, 
$$
\vdash A_{q}(\qqqq)\mbox{ in } S
$$
yields
$$
\mbox{not }\vdash \neg A_{q}(\qqqq)\mbox{ in } S.
$$
Therefore for any non-negative integer $d$, $\HHH(q,d)$ is false. In particular $\HHH(q,0)$, $\cdots,$ $\HHH(q,e)$ are all false. Thus by the numeralwise expressibility of the predicate $\HHH(a,c)$, we have
$$
\vdash \neg h(\qqqq,\ooo),\ \vdash \neg h(\qqqq,\llll),\ \cdots,\ \vdash \neg h(\qqqq,\eee).
$$
Hence
$$
\vdash \forall c\hskip1pt (c\le \eee \Rightarrow \neg h(\qqqq,c)).
$$
This with $\vdash g(\qqqq,\eee)$ in \eq{Rosser1} gives
$$
\vdash \exists b 
\left( g(\qqqq,b)\hskip1pt\land \hskip0pt\forall c\hskip1pt(c\le b\Rightarrow\neg h(\qqqq,c))
\right).
$$
This is equivalent to
$$
\vdash \neg A_{q}(\qqqq)\mbox{ in } S.
$$
This and \eq{(5)} imply that $S$ is inconsistent, which contradicts our premise that $S$ is consistent. Therefore we have
$$
\mbox{not } \vdash A_{q}(\qqqq)\mbox{ in } S.
$$


On the other hand let us suppose that
\beq
\vdash \neg A_{q}(\qqqq)\mbox{ in } S.\label{(7)}
\ene
Then there is a G\"odel number $k$ of a proof of $\neg A_{q}(\qqqq)$ in $S$, and $\HHH(q,k)$ is true. Therefore by the numeralwise expressibility of $\HHH(a,c)$ we have
$$
\vdash h(\qqqq,\kkk).
$$
From this follows
\beq
\vdash \forall b\left(b\ge \kkk\Rightarrow \exists c\hskip1pt(c\le b\hskip1pt\land \hskip0pt  h(\qqqq,c))\right). \label{(8)}
\ene
As we have assumed that $\neg A_{q}(\qqqq)$ is provable in $S$, from our premise that $S$ is consistent, there is no proof of $A_{q}(\qqqq)$ in $S$. Thus
\beqs
\vdash \neg g(\qqqq,\ooo),\ \vdash \neg g(\qqqq,\llll),\ \cdots,\ \vdash \neg g(\qqqq,\kkk-\llll).
\enes
Therefore
\beq
\vdash \forall b\left(b<\kkk\Rightarrow \neg g(\qqqq,b)\right) .\nonumber
\ene
Combining this with \eq{(8)} yields
$$
\vdash \forall b\left( \neg g(\qqqq,b)\vee \exists c\hskip1pt(c\le b\hskip1pt\land \hskip0pt  h(\qqqq,c))\right).
$$
This is equivalent to
$$
\vdash A_{q}(\qqqq).
$$
This and \eq{(7)} imply that $S$ is inconsistent, which contradicts our premise that $S$ is consistent. Therefore we obtain
\beq
\mbox{not }\vdash \neg A_{q}(\qqqq)\mbox{ in } S.\nonumber
\ene
\end{proof}



It thus suffices to prove Theorem \ref{th7.7} in order to prove Theorem \ref{th7.9}. 


\section{Recursiveness}\label{chap:8}

\normalsize

\Large

\subsection{Recursive functions}\label{8.1}


\normalsize

Until now we have assumed that `recursive' means `inductive' in somewhat vague manner. In this section we define recursiveness rigorously. In this paper the functions whose domain and range are subsets of $\N$ are called the number-theoretic functions.

\begin{df}\label{df8.2}{\rm A function $\phi=\phi(x_1,\dots,x_n)$ is called a primitive recursive function if it is defined by repetitive applications of the following equations I) -- V). Here it is assumed that $n,m\ge1$ are integers, $i$ is the integer such that $1\le i\le n$, and $q$ is a natural number. Further $\psi,\chi,\chi_1,\dots,\chi_m$ are number-theoretic functions which have the indicated number of variables.
\begin{enumerate}
\item[   I)] $\phi(x)=x'$.
\item[  II)] $\phi(x_1,\dots,x_n)=q$.
\item[ III)] $\phi(x_1,\dots,x_n)=x_i$.
\item[  IV)] $\phi(x_1,\dots,x_n)=\psi(\chi_1(x_1,\dots,x_n),\dots,\chi_m(x_1,\dots,x_n))$.
\item[  V)] 
\begin{enumerate}
\item[ 1)]
When $n=1$

$\phi(0)=q$,\quad $\phi(k+1)=\chi(k,\phi(k))$.
\item[ 2)]When $n\ge 2$
\begin{enumerate}
\item[  i)] $\phi(0,x_2,\dots,x_n)=\psi(x_2,\dots,x_n)$.
\item[ ii)] $\phi(k+1,x_2,\dots,x_n)=\chi(k,\phi(k,x_2,\dots,x_n),x_2,\dots,x_n)$.
\end{enumerate}
\end{enumerate}
\end{enumerate}}
\end{df}

In set-theoretic number theory, natural number $n(\ge0)$ is defined as follows as stated before. First $0$ is defined as the empty set $\emptyset$, $1$ is defined as the successor $0\cup\{0\}$ of $0$, namely as the set $\{0\}$, $2$ is defined as the successor $\{0,1\}$ of $1$, and so on. A general natural number $n$ is thus defined similarly as $n=\{0,1,\dots,n-1\}$. Therefore if, for a natural number $n\ge0$, a number-theoretic function $F$ is defined for the natural numbers $0,1,\dots,n-1$ less than $n$, the function $F|n$ which is a restriction of $F$ to the domain $n=\{0,1,\dots,n-1\}\in \N$ is defined. In this case if some number-theoretic function $G$ is given, we can construct recursively a number-theoretic function $F$ whose domain is $\N$ by
$$
F(n)=G(n,F|n).
$$
In general the function constructed in this way is called a recursive function. It is easy to see that the functions defined in definition \ref{df8.2} can be written in this form.
\SP

\begin{df}\label{df8.3}{\rm A function $f: A \longrightarrow B$ whose domain is a subset of $A$ and range is a subset of $B$ is called a total function, if its domain $\DD(f)$ is equal to $A$. If otherwise the value $f(x)$ is not defined for an element $x$ of $A$, $f$ is called a partial function.}
\end{df}

\begin{df}\label{df8.4}{\rm For a number-theoretic function $\psi(x_1,\dots,x_n,y)$, we define the $\mu$-operator by
\beq
&\hskip-40pt&\mu y[\psi(x_1,\dots,x_n,y)=0]=y_0\nonumber\\
&\hskip-40pt\stackrel{\scriptsize\mbox{{\it def}}}\Leftrightarrow&
\psi(x_1,\dots,x_n,y_0)=0\wedge\nonumber (\forall y<y_0)[\psi(x_1,\dots,x_n,y)\ne 0].\label{mu}
\ene
}
\end{df}

From an intuitionistic viewpoint, $\mu$-operator is defined as follows. With calculating $\psi(x_1,$ $\dots,$ $x_n,$ $0)$, $\psi(x_1$, $\dots,x_n,$ $1)$, $\psi(x_1$, $\dots$, $x_n,$ $2)$, $\dots$, one searches for a natural number $y$ which satisfies $\psi(x_1,\dots,x_n,y)=0$. If he finds one such $y$, the first $y$ is $y_0=\mu y[\psi(x_1,\dots,x_n$, $y)=0]$. If otherwise there is no such $y$, this search will continue forever, and for $(x_1,\dots,x_n)$, the value $\mu y[\psi(x_1,\dots,x_n,y)=0]$ is not defined. Therefore a function defined by $\mu$-operator can be a partial function. The case we are interested in is that the function is a total function. In this case
\begin{enumerate}
\item[ VI)] $\phi(x_1,\dots,x_n)=\mu y[\psi(x_1,\dots,x_n,y)=0]$
\end{enumerate}
defines a new number-theoretic total function $\phi(x_1,\dots,x_n)$.

\begin{df}\label{df8.5}{\rm A number-theoretic total function $\phi$ is called a recursive function if there is a sequence of number-theoretic total functions $\phi_1,\phi_2,\dots,\phi_n$ such that the last function $\phi_n$ is $\phi$ and each function $\phi_k$ in the sequence is a function defined by I) -- III) or obtained by applying IV) -- VI) to the former functions $\phi_1,\phi_2,\dots,\phi_{k-1}$. The minimum length $n$ of such a sequence $\phi_1,\phi_2,\dots,\phi_n$ for a function $\phi=\phi_n$ is called the degree of the function $\phi$.
}
\end{df}


\Large

\subsection{Recursive relations}\label{8.2}


\normalsize

Using the recursive function defined in the previous subsection, we define recursive predicates or recursive relations.
\SP

\begin{df}\label{df8.6}{\rm A relation among natural numbers or a predicate about natural numbers $\RRR(x_1,\dots,x_n)$ is called recursive if there exists a recursive function $\phi(x_1$, $\dots$, $x_n)$ such that for any natural numbers $x_1,\dots,x_n$
\beq
\RRR(x_1,\dots,x_n)\Leftrightarrow [\phi(x_1,\dots,x_n)=0]
\ene
holds.}
\end{df}

Functions $x+y$, $x\cdot y$, $x^y$ are recursive, and the relation $x=y$ is also recursive.

\begin{df}\label{df8.7}{\rm For an expression $\EEE$, terms $t, t_1, t_2$ and a variable $x$, we define
\begin{enumerate}
\item[ 1)] $t_1\ne t_2\stackrel{\scriptsize\mbox{{\it def}}}=\neg (t_1=t_2)$.
\item[ 2)] $t_1\le t_2\stackrel{\scriptsize\mbox{{\it def}}}=\exists x(t_1+x=t_2)$.
\item[ 3)] $(\forall x\le t)\EEE\stackrel{\scriptsize\mbox{{\it def}}}=\forall x(x\le t\Rightarrow \EEE)$,

$(\exists x\le t)\EEE\stackrel{\scriptsize\mbox{{\it def}}}=\neg (\forall x\le t)\neg \EEE$.
\item[ 4)] $t_1<t_2\stackrel{\scriptsize\mbox{{\it def}}}=(t_1\le t_2)\wedge (t_1\ne t_2)$.
\item[ 5)] $(\forall x< t)\EEE\stackrel{\scriptsize\mbox{{\it def}}}=\forall x(x< t\Rightarrow \EEE)$,

$(\exists x< t)\EEE\stackrel{\scriptsize\mbox{{\it def}}}=\neg (\forall x<t)\neg \EEE$.\end{enumerate}}
\end{df}

If an expression $\EEE$ defines a recursive predicate, the relations in the definition \ref{df8.7} define recursive relations. In fact, the following is known (S\"atze II - IV in G\"odel \cite{[G]}). 

\begin{thm}\label{th8.8}{\sl 
\begin{enumerate}
\item[ (1)] If the predicates $\RRR$ and $\SSSS$ are recursive, the predicates $\neg \RRR$, $\RRR\wedge\SSSS$, $\RRR\lor\SSSS$, $\RRR\Rightarrow\SSSS$ are also recursive.
\item[ (2)] If functions $\phi(x_1,\dots,x_n)$ and $\psi(x_1,\dots,x_n)$ are recursive, then the relation or predicate $\phi(x_1$, $\dots$, $x_n)$ $=$ $\psi(x_1$, $\dots$, $x_n)$ is also recursive.
\item[ (3)] If a function $\phi(x_1,\dots,x_n)$ and a predicate $\RRR(u,y_1,\dots,y_m)$ are recursive, the predicates $\SSSS$, $\TTTT$ and the function $\psi$ defined below are also recursive.
\begin{enumerate}
\item[   i)] $\SSSS(x_1,\dots,x_n,y_1,\dots,y_m)$

$\stackrel{\scriptsize\mbox{{\it def}}}=(\exists u)\left[\left(u\le \phi(x_1,\dots,x_n)\right)\wedge\RRR(u,y_1,\dots,y_m)\right]$,
\item[  ii)] $\TTTT(x_1,\dots,x_n,y_1,\dots,y_m)$

$\stackrel{\scriptsize\mbox{{\it def}}}=(\forall u)\left[\left(u\le\phi(x_1,\dots,x_n)\right)\Rightarrow\RRR(u,y_1,\dots,y_m)\right]$,
\item[ iii)] $\psi(x_1,\dots,x_n,y_1,\dots,y_m)$

$\stackrel{\scriptsize\mbox{{\it def}}}=\hskip-2pt\mu u\left[\left(u\le \phi(x_1,\dots,x_n)\right)\wedge\RRR(u,y_1,\dots,y_m)\right]$.

Here if there is no least natural number $u$ which satisfies the condition in the parentheses $[\ \ ]$, we define the right hand side to be 0\footnote{The condition in $[\ \ ]$ on the right hand side is determined by searching for a finite number of natural numbers $u$. Thus this function $\psi$ is a total function.}.
\end{enumerate}
\end{enumerate}}
\end{thm}

\section{Numeralwise expression of proof}\label{chap:9}

\normalsize


In the next section we will show that the predicates $\GGG(a,b)$, $\HHH(a,b)$ defined in section \ref{chap:7} are numeralwise expressible in number theory $S$. To show this G\"odel \cite{[G]} proved the following theorem, and used the fact that the predicates $\GGG(a,b)$, $\HHH(a,b)$ are recursive.

\begin{thm}\label{th9.1}{\sl For any recursive relation $\RRR(x_1,\dots,x_n)$ there exists a number-theoretic formula $r(u_1,\dots,u_n)$ with $n$ free variables $u_1,\dots,u_n$ such that for any $n$-tuple of natural numbers $x_1,\dots,x_n$ the following i) and ii) hold.
\begin{enumerate}
\item[ i)] If $\RRR(x_1,\dots,x_n)$ is true, then $\vdash r(\xxxa,\dots,\xxxn)$holds.
\item[ ii)] If $\RRR(x_1,\dots,x_n)$ is false, then $\vdash \neg\hskip1pt r(\xxxa,\dots,\xxxn)$ holds.
\end{enumerate}}
\end{thm}

In this paper we do not prove this theorem. Instead we will prove directly that the predicates $\GGG(a,b)$, $\HHH(a,b)$ are numeralwise expressed by some formulae $g(a,b)$, $h(a,b)$, respectively.


\subsection{Numeralwise expression of being terms and formulae}\label{9.1}


\normalsize

In this section we will give a numeralwise expression of the predicate ``A given expression $E_x$ is a proof." We recall the correspondence of G\"odel numbers to primitive symbols.
\beqs
\hskip-22pt&&\begin{array}{cccccccccc}
{}^\prime  &
  0 &
 ( &
 ) &
 \{ &
 \} &
 [ &
 ] &
 + &
 \cdot \\
2^{0}  & 2^{1} & 2^{2}  & 2^{3} & 2^{4} & 2^{5} & 2^{6}  & 2^{7}  & 2^{8} & 2^{9}
\end{array}\\
\hskip-22pt&&\begin{array}{cccccccccc}
 = &
 \Rightarrow &
 \wedge &
 \lor &
 \neg &
 \forall &
 \exists &
 ,  \\
 2^{10} & 2^{11} & 2^{12} & 2^{13} & 2^{14} & 2^{15} & 2^{16} & 2^{17}
\end{array}
\enes

The following procedure is a variant of the procedure stated in \cite{Kitada-book}.

First of all we will show that it is possible to express the procedure of constructing G\"odel number in a recursive way. For this purpose it suffices to prove that, for any given natural numbers $x,y,z$, it is possible to express the fact that $z$ is equal to the product $x\star y$ by a proposition in the number theory $S$ in a recursive way. Let us recall, as stated in section \ref{chap:8}, that the functions $x+y$, $x\cdot y$, $x^y$ are recursive functions and the relation $x=y$ is a recursive relation. By definition \ref{df8.7} and theorem \ref{th8.8}, the following definitions {\bf 1} -- {\bf 28} are all recursive.

\begin{description}
\item[1.] Div$(x,y)$ : $x$ is a factor of $y$.
\begin{equation*}
(\exists z \leq y ) \left( x\cdot z =y \right) 
\end{equation*}
\item[2.] $2^{\times }(x)$ : $x$ is a power of $2$.
\begin{equation*}
(\forall z \leq x ) \bigl( \left( \mbox{Div}(z,x)    \wedge (z\neq 1) \right) \Rightarrow \mbox{Div}(2,z)  \bigr) 
\end{equation*}
\item[3.] $y=2^{\ell(x)}$ : $y$ is the least power of $2$ which is greater than $x$.
\begin{eqnarray*}
&&\hskip-20pt\left(2^{\times } (y) \wedge (y > x) \wedge (y>1) \right)\wedge(\forall z < y ) \neg \left( 2^{\times } (z)  \land (z>x) \land (z>1)  \right) 
\end{eqnarray*}
\item[4.] $z=x\star y $ : $z$ is the numeral resulting from the $\star$-product of $x$ and $y$.
\begin{equation*}
(\exists w \leq z )
(z = \left( w\cdot x\right) + y \land w=2^{\ell(y)})
\end{equation*}
\end{description}

We next decompose numerals expressed in binary numbers, and express the procedure to extract a subsequence from a sequence of primitive symbols in number-theoretic way.
\begin{description}
\item[5.] Begin$(x,y)$ : $x$ is the numeral which expresses a left-most part of the sequence of symbols which has G\"odel number $y$.
\begin{equation*}
x=y \lor \left(  x\neq 0 \land 
(\exists z \leq y ) \left( x\star z=y \right) 
\right) 
\end{equation*}
\item[6.] End$(x,y)$ : $x$ is the numeral which expresses a right-most part of the sequence of symbols which has G\"odel number $y$.
\begin{equation*}
x=y \lor \bigl(  x\neq 0 \land (\exists z \leq y ) \left( z\star x=y \right) \bigr) 
\end{equation*}
\item[7.] Part$(x,y)$ : $x$ is the numeral which expresses a part of the sequence of symbols which has G\"odel number $y$.
\begin{equation*}
\hskip-18ptx=y \lor \bigl(  x\neq 0 \land (\exists z \leq y ) \left( \mbox{End}(z,y) \wedge \mbox{Begin}(x,z) \right) \bigr) 
\end{equation*}
\end{description}
Using these, we can construct a predicate which classifies the nature of terms.
\begin{description}
\item[8.] Succ$(x)$ : $E_x$ is a sequence of ${}^\prime $.
\begin{equation*}
(x\neq 0) \land (\forall y \leq x ) \left( \mbox{Part}(y,x) \Rightarrow \mbox{Part}(1,y)
\right) 
\end{equation*}
\item[9.] Var$(x)$ : $E_x$ is a variable.
\begin{equation*}
(\exists y \leq x ) \left( \mbox{Succ}(y) \wedge x=2^2\star 2^1\star y \star 2^3 \right) 
\end{equation*}
Her we recall that we follow the convention stated in subsection \ref{7.2} such that the variables $a,b,c,\dots$ are supposed to be expressed as $(0')$, $(0^{\prime\prime})$, $(0^{\prime\prime\prime})$, $\dots$.

\item[10.] Num$(x)$ : $E_x$ is a numeral.
\begin{equation*}
(x=2^1) \lor (\exists y \leq x ) \left( \mbox{Succ}(y) \wedge x=2^1\star y  \right) 
\end{equation*}
\end{description}

G\"odel number of a sequence of (formal) expressions $E_{x_1}$, $E_{x_2}$, $\dots$, $E_{x_n}$ is written as follows.
\begin{eqnarray*}
x_1 \star 2^{17} \star  x_2 \star  2^{17} \star  ...\star
2^{17} \star   x_n 
\end{eqnarray*}
The facts that an expression is a sequence of formal expressions and that an expression is included in a sequence of expressions are expressed by the following propositions.
\begin{description}
\item[11.] Seq$(x)$ : $E_x$ is a sequence of formal expressions.
\begin{equation*}
\mbox{Part}( 2^{17} , x) 
\end{equation*}
\item[12.] $x\in y $ : $E_y$ is a sequence of expressions, and $E_x$ is an element of it.
\beqs
&&\hskip-60pt\mbox{Seq}(y)  \wedge \neg \mbox{Part}(2^{17}, x)\wedge\\
&&\hskip-29pt\biggl( \mbox{Begin}( x \star 2^{17} , y) \lor \mbox{End}( 2^{17}  \star x, y)\lor\mbox{Part}( 2^{17}\star x \star 2^{17} , y)\biggr)
\enes
\item[13.] $x \prec_z y$ : For two sequences $E_x$, $E_y$ of expressions which are elements of a sequence $E_z$ of expressions, $E_x$ appears before $E_y$.
\begin{equation*}
\hskip-34pt(x\in z) \land (y\in z) \land
(\exists w \leq z) \mbox{Part}( x \star w \star y , z)
\end{equation*}
\end{description}
Using those, the fact that an expression which has G\"odel number $x$ is a formula is expressed in the formal system $S$.
\begin{description}
\item[14.] Term$(x)$ : $E_x$ is a term.
\begin{eqnarray*}
&\hskip-40pt & \exists y \biggl( ( x\in y)  \land  (\forall z\in y) \bigl\{\mbox{Var}(z)  \lor \mbox{Num}(z)  \lor 
\\
&\hskip-40pt & (\exists v \prec_y z )  (\exists w \prec_y z ) \bigl[
(2^2 \star v \star 2^3 \star 2^8 \star 2^2 \star w \star 2^3 = z )  \lor
\\
&\hskip-40pt &  (2^2 \star v \star 2^3 \star 2^9 \star 2^2 \star  w \star 2^3  = z ) \lor (2^2 \star v \star 2^3 \star 2^0 = z )
\bigr] \bigr\} \biggr)
\end{eqnarray*}
\item[15.] Atom$(x)$ : $E_x$ is an atomic formula.
\begin{eqnarray*}
& \hskip-50pt & (\exists y \leq x ) (\exists z \leq x )  \biggl(  \mbox{Term}( y)  \land 
\mbox{Term}( z ) \land
\left(  (x= y \star 2^{10}  \star  z ) 
\lor (x= \mbox{leq}( y , z)  ) 
\right) \biggr)
\end{eqnarray*}
Here the function leq is defined as follows recursively.
\begin{enumerate}
\item neq$(x,y)$ is the following G\"odel number of the expression $E_x\neq E_y$.
\begin{eqnarray*}
2^{14}  \star 2^2  \star  x  \star 2^{10}  \star  y  \star 2^3
\end{eqnarray*}
\item leq$(x,y)$ is the following G\"odel number of the expression $E_x\leq E_y$.
\normalsize
\begin{eqnarray*}
& \hskip-40pt & 2^{14} \star 2^2 \star 2^{15} \star 2^2 \star 2^1 \star 2^0 \star 2^3 \star 2^2 \star
\\
& \hskip-40pt &  \mbox{neq}( x \star 2^{8} \star 2^2 \star 2^1 \star 2^0 \star 2^3, y) \star 2^3 \star 2^3
\end{eqnarray*}
\normalsize
\end{enumerate}
\item[16.]  Gen$(x,y)$ : For a variable $E_u$, $E_y$ is equal to $ \forall E_u (E_x)$.
\begin{eqnarray*}
(\exists u \leq y )\left( \mbox{Var}(u) \land
y=2^{15} \star u \star 2^2\star x \star 2^3
\right)
\end{eqnarray*}
\item[17.] Form$(x)$ : $E_x$ is a (well-formed) formula.
\begin{eqnarray*}
&\hskip-40pt & \exists y \biggl(  ( x\in y)  \land  (\forall z\in y) \big\{
\mbox{Atom}(z)  \lor 
\\
&\hskip-40pt & (\exists v \prec_y z)  (\exists w \prec_y z ) \bigl[( z= v\star 2^{11} \star w ) \lor
\left( z=2^{14} \star 2^2 \star v \star 2^3 \right)
\lor \mbox{Gen}(w,z) 
 \bigr]   \bigr\}  \biggr)
\end{eqnarray*}
Here we regard the logical symbols of propositional calculus consisting of only $\neg$ and $\Rightarrow$ with noting that logical symbols $\land$ and $\lor$ are expressed by using $\neg$ and $\Rightarrow$ as follows:
\beqs
&&\hskip-20ptA\land B\mbox{ is }\neg(A\Rightarrow\neg B),\\
&&\hskip-20ptA\lor B\mbox{ is }\neg A\Rightarrow B
\enes
As well we regard that the existential quantifier is expressed as follows with using the universal quantifier:
$$
\exists x F(x)\mbox{ is }\neg\forall x\neg  F(x)
$$
\end{description}


\subsection{Numeralwise expression of being axioms of propositional calculus}\label{9.2}


\normalsize

We next show that the fact that an expression with G\"odel number $x$ is an axiom of number theory is expressed by a formula in $S$. First we consider the axioms of propositional calculus.
\begin{description}
\item[18.] Pro$(x)$:  $E_x$ is an axiom of propositional calculus.
\begin{eqnarray*}
& \hskip-40pt &\mbox{Prop}_1(x)\lor \mbox{Prop}_2(x)\lor \mbox{Prop}_3(x)\lor \mbox{Prop}_4(x)\lor
 \mbox{Prop}_5(x)\lor\mbox{Prop}_6(x)\lor \\
& \hskip-40pt &
\mbox{Prop}_7(x)\lor \mbox{Prop}_8(x)\lor
\mbox{Prop}_9(x)\lor \mbox{Prop}_{10}(x)\lor \mbox{Prop}_{11}(x)
\end{eqnarray*}
Here $\mbox{Prop}_1(x)$, $\mbox{Prop}_2(x)$, $\mbox{Prop}_3(x)$, $\mbox{Prop}_4(x)$, $\mbox{Prop}_5(x)$, $\mbox{Prop}_6(x)$, \linebreak$\mbox{Prop}_7(x)$, $\mbox{Prop}_8(x)$, $\mbox{Prop}_9(x)$, $\mbox{Prop}_{10}(x)$, $\mbox{Prop}_{11}(x)$ are defined as follows.

\begin{enumerate}
\item $\mbox{Prop}_1(x)$ : $E_x$ is axiom 1 of propositional calculus.
\normalsize
\begin{eqnarray*}
&\hskip-40pt& (\exists a<x) (\exists b<x) (
\mbox{Form}(a) \land \mbox{Form}(b) \land x =   a \star  2^{11}\star 2^2 \star  b \star 2^{11}\star  a \star 2^3 )
\end{eqnarray*}
\normalsize
\item $\mbox{Prop}_2(x)$ : $E_x$ is axiom 2 of propositional calculus.
\normalsize
\begin{eqnarray*}
&&\hskip-66pt  (\exists a<x) (\exists b<x) (\exists c<x) (
\mbox{Form}(a) \land \mbox{Form}(b)\land\\
&&\hskip-30pt \mbox{Form}(c) \land
  x = 2^2\star  a \star 2^{11}\star b\star 2^3\star
2^{11}\star 2^2\star 2^2\star a\star 2^{11}\\
&&\hskip-30pt\star 2^2\star b\star 2^{11}\star c\star 2^3\star 2^3\star 2^{11}\star 2^2\star a\star 2^{11}\star c\star 2^3\star 2^3)
\end{eqnarray*}
\normalsize
\item $\mbox{Prop}_3(x)$ : $E_x$ is axiom 3 of propositional calculus.
\normalsize
\begin{eqnarray*}
&&\hskip-80pt(\exists a<x) (\exists b<x) (
\mbox{Form}(a) \land \mbox{Form}(b) \land 
\\
&&\hskip-40pt
x =  a \star 2^{11}\star 2^2\star2^2\star a\star 2^{11}\star b \star 2^3\star 2^{11}\star b\star 2^3)
\end{eqnarray*}
\normalsize
\item $\mbox{Prop}_4(x)$ : $E_x$ is axiom 4 of propositional calculus.
\normalsize
\begin{eqnarray*}
&&\hskip-100pt (\exists a<x) (\exists b<x) (\mbox{Form}(a) \land \mbox{Form}(b) \land 
\\
&&\hskip-40pt
x = a \star 2^{11}\star 2^2\star b\star 2^{11}\star a\star 2^{12}\star b\star 2^3)
\end{eqnarray*}
\normalsize
\item $\mbox{Prop}_5(x)$ : $E_x$ is axiom 5 of propositional calculus.
\normalsize
\begin{eqnarray*}
&&\hskip-32pt (\exists a<x) (\exists b<x)  (\exists c<x)(
\mbox{Form}(a) \land \mbox{Form}(b)\wedge x =   a \star 2^{12}\star b\star 2^{11}\star  a)
\end{eqnarray*}
\normalsize
\item $\mbox{Prop}_6(x)$ : $E_x$ is axiom 6 of propositional calculus.
\normalsize
\begin{eqnarray*}
&\hskip-32pt& (\exists a<x) (\exists b<x)  (\exists c<x)(
\mbox{Form}(a) \land \mbox{Form}(b)\wedge x =   a \star 2^{12}\star b\star 2^{11}\star b)
\end{eqnarray*}
\normalsize
\item $\mbox{Prop}_7(x)$ : $E_x$ is axiom 7 of propositional calculus.
\normalsize
\begin{eqnarray*}
&\hskip-40pt& (\exists a<x) (\exists b<x)(
\mbox{Form}(a) \land\mbox{Form}(b)\land x = a \star 2^{11}\star a\star 2^{13}\star b)
\end{eqnarray*}
\normalsize
\item $\mbox{Prop}_8(x)$ : $E_x$ is axiom 8 of propositional calculus.
\normalsize
\begin{eqnarray*}
&\hskip-40pt& (\exists a<x) (\exists b<x) (
\mbox{Form}(a) \land\mbox{Form}(b)\land x = b \star 2^{11}\star a\star 2^{13}\star b)
\end{eqnarray*}
\normalsize
\item $\mbox{Prop}_9(x)$ : $E_x$ is axiom 9 of propositional calculus.
\normalsize
\begin{eqnarray*}
&&\hskip-26pt (\exists a<x) (\exists b<x) (\exists c<x) (
\mbox{Form}(a) \land\mbox{Form}(b)\land \mbox{Form}(c)\land\\
&&\hskip-10pt x = 2^2\star a\star 2^{11} \star c\star 2^3\star 2^{11}\star\ 2^2\star 2^2\star b\star 2^{11}\star c\star 2^3\star 2^{11}\star 2^2\\
&&\hskip105pt \star\ a\star 2^{13}\star b\star 2^{11}\star c\star 2^3\star 2^3)
\end{eqnarray*}
\normalsize
\item $\mbox{Prop}_{10}(x)$ : $E_x$ is axiom 10 of propositional calculus.
\normalsize
\begin{eqnarray*}
&&\hskip-50pt (\exists a<x) (\exists b<x) (
\mbox{Form}(a) \land\mbox{Form}(b)\land\\
&&\hskip-20pt x = 2^2\star a\star 2^{11} \star b\star 2^3\star 2^{11}\star 2^2\star 2^2\star a\star 2^{11}\star 2^{14}\star\\
&&\hskip20pt b\star 2^3\star 2^{11}\star 2^{14}\star a\star 2^3)
\end{eqnarray*}
\normalsize
\item $\mbox{Prop}_{11}(x)$ : $E_x$ is axiom 11 of propositional calculus.
\normalsize
\begin{eqnarray*}
&&\hskip-70pt (\exists a<x)  (
\mbox{Form}(a) \land x = 2^{14}\star 2^{14}\star a\star 2^{11} \star a)
\end{eqnarray*}
\normalsize
\end{enumerate}
\end{description}


\Large

\subsection{Numeralwise expression of being axioms of predicate calculus}\label{9.3}


\normalsize

We now express in $S$ that an expression $E_x$ is an axiom of predicate calculus. As is easily seen, axiom 1 and axiom 4 are equivalent, and axiom 2 and axiom 3 are equivalent. Thus we have only to give expressions to axioms 1 and 2. In axiom 2, we need to consider the replacement of all occurrences of a free variable in a formula by a term.
\begin{description}
\item[19.] $\mbox{Free}(x,y)$ : Every variable in a term $E_x$ is not bounded in an expression $E_y$.
\begin{eqnarray*}
&&\hskip-40pt \mbox{Term}(x) \land (\forall z < x )\biggl(
\left[\mbox{Var}(z) \land  \mbox{Part}(z,x) \right] \Rightarrow
\left[ \neg \mbox{Part}(2^{15}\star z,y) \right]
\biggr)
\end{eqnarray*}
\item[20.] $\mbox{Pred}_1(x)$ : $E_x$ is axiom 1 of predicate calculus.
\begin{eqnarray*}
&\hskip-45pt& (\exists a < x ) (\exists b < x )  (\exists c < x )
\bigl(\mbox{Form}(a) \land \mbox{Form}(b) \land \mbox{Var}(c) \land \\
&\hskip-45pt&  (\neg \mbox{Part}(c,b))\land
 x= 2^2\star b \star 2^{11} \star a \star 2^3\star
2^{11}\star 2^2\star b  
\star 2^{11} \star 2^2 \star 2^{15}\star c\star a \star 2^3\star 2^3\bigr)
\end{eqnarray*}
\item[21.] $\mbox{Seq}(x,y,u)$ : An expression $u$ includes a pair $\{E_x,E_y\}$ as a consecutive pair of $E_x$ and $E_y$ in this order.
\begin{eqnarray*}
&&\hskip-60pt\neg \mbox{Seq}(x)\land \neg \mbox{Seq}(y)  \land(x\ne0)\land(y\ne0)\land\mbox{Part}(x \star 2^{17}\star y,u)
\end{eqnarray*}
\item[22.] $x=\mbox{alt}_y(u,t)$ : A formula $E_x$ is obtained from a formula $E_y$ by substituting a free term $E_t$ at every occurrence of a free variable $E_u$.
\begin{eqnarray*}
&& \hskip-40pt  \mbox{Form}(x) \land \mbox{Form}(y) \land \mbox{Var}(u)\land
\mbox{Free}(u,y)\land\mbox{Term}(t)\land\mbox{Free}(t,y)\land\mbox{Part}(u, y)\land\\
&& \hskip-40pt  \neg \mbox{Part}(u, x)\land\exists w \biggl\{\mbox{Seq}(y,x,w)  \land 
 (\forall a <w )(\forall b <w )\biggl(
\mbox{Seq}(a,b,w) \\
&& \hskip-0pt \Rightarrow
\bigl\{(\neg\mbox{Part}(u, a) \land a=b)\lor
 (\exists c_1 < a)(\exists c_2 < b) (\exists d_1 < a)(\exists d_2 < b)\\
&& \hskip-0pt \bigl[ \mbox{Seq}(c_1,c_2,w) \land 
\mbox{Seq}(d_1,d_2,w)\land  a = c_1 \star u \star d_1 \land  b = c_2 \star t \star d_2\bigr]
\bigr\}\biggr)\biggr\}
\end{eqnarray*}
\item[23.] $\mbox{Pred}_2(x)$ : $E_x$ is axiom 2 of predicate calculus.
\normalsize
\begin{eqnarray*}
&& \hskip-40pt  (\exists a < x ) (\exists b < x )  (\exists c < x )(\exists t < x )
(\mbox{Form}(a) \land \\
&& \hskip-0pt \mbox{Var}(b) \land
 \mbox{Term}(t) \land
 c=\mbox{alt}_a(b,t) \land x= 2^{15}\star b\star a\star 2^{11}\star c)
\end{eqnarray*}
\normalsize
\end{description}

\Large

\subsection{Numeralwise expression of being axioms of number theory}\label{9.4}


\normalsize

Finally we express that $E_x$ is an axiom of number theory.

\begin{description}
\item[24.] $\mbox{Nat}(x)$: $E_x$ is an axiom of number theory.
\begin{eqnarray*}
&\hskip-40pt&\mbox{Nat}_1(x)\hskip-1.5pt\lor\hskip-1.5pt \mbox{Nat}_2(x)\hskip-1.5pt\lor\hskip-1.5pt \mbox{Nat}_3(x)\hskip-1.5pt\lor\hskip-1.5pt \mbox{Nat}_4(x)\hskip-1.5pt\lor\hskip-1.5pt \mbox{Nat}_5(x)\hskip-1.5pt\lor\hskip-1.5pt
 \mbox{Nat}_6(x)\hskip-1.5pt\lor\hskip-1.5pt \mbox{Nat}_7(x)\hskip-1.5pt\lor\hskip-1.5pt \mbox{Nat}_8(x)
\end{eqnarray*}
Here $\mbox{Nat}_1(x)$, $\mbox{Nat}_2(x)$, $\mbox{Nat}_3(x)$, $\mbox{Nat}_4(x)$, $\mbox{Nat}_5(x)$, $\mbox{Nat}_6(x)$, $\mbox{Nat}_7(x)$, $\mbox{Nat}_8(x)$ are defined as follows.
\begin{enumerate}
\item $\mbox{Nat}_1(x)$ : $E_x$ is axiom 1 of number theory.
\begin{eqnarray*}
&&\hskip-40pt (\exists a < x )(\exists b < x )
(\mbox{Term}(a) \land \mbox{Term}(b) \land \\
&&\hskip-20pt x=2^2 \star a \star 2^0\star 2^{10}
\star b \star 2^0\star 2^3\star 2^{11}\star 2^2  \star a \star 2^{10} \star b \star 2^3)
\end{eqnarray*}
\item $\mbox{Nat}_2(x)$ : $E_x$ is axiom 2 of number theory.
\begin{eqnarray*}
&\hskip-40pt&  (\exists a < x )
(\mbox{Term}(a) \land   x=2^{14}\star 2^2 \star a \star
2^0\star 2^{10}\star 2^1\star 2^3)
\end{eqnarray*}
\item $\mbox{Nat}_3(x)$ : $E_x$ is axiom 3 of number theory.
\begin{eqnarray*}
&& \hskip-40pt  (\exists a < x )(\exists b < x )(\exists c < x )
(\mbox{Term}(a) \land \mbox{Term}(b) \land \mbox{Term}(c) \land\\
&& \hskip-20pt 
 x= a \star 2^{10}\star b \star 2^{11} \star 2^2\star a \star 2^{10} \star c \star 2^{11}
\star b \star 2^{10}  \star c \star 2^3)
\end{eqnarray*}
\item $\mbox{Nat}_4(x)$ : $E_x$ is axiom 4 of number theory.
\begin{eqnarray*}
&& \hskip-75pt  (\exists a < x )(\exists b < x )
(\mbox{Term}(a) \land \mbox{Term}(b) \land \\
&& \hskip-20pt  x= a \star 2^{10}\star b \star 2^{11}  \star a \star 2^0\star 2^{10} \star b \star 2^0)
\end{eqnarray*}
\item $\mbox{Nat}_5(x)$ : $E_x$ is axiom 5 of number theory.
\begin{eqnarray*}
&& \hskip-110pt  (\exists a < x )
(\mbox{Term}(a) \land  x=a \star 2^8\star2^1\star2^{10}
\star a)
\end{eqnarray*}
\item $\mbox{Nat}_6(x)$ : $E_x$ is axiom 6 of number theory.
\begin{eqnarray*}
&& \hskip-60pt  (\exists a < x )(\exists b < x )
(\mbox{Term}(a) \land \mbox{Term}(b)\land\\
&& \hskip-20pt  x=a \star 2^8 \star b\star 2^0\star 2^{10}\star 2^2\star a\star 2^8\star b\star 2^3\star 2^0)
\end{eqnarray*}
\item $\mbox{Nat}_7(x)$ : $E_x$ is axiom 7 of number theory.
\begin{eqnarray*}
&& \hskip-110pt  (\exists a < x )
(\mbox{Term}(a) \land x=a \star 2^9 \star 2^1\star 2^{10}\star 2^1)
\end{eqnarray*}
\item $\mbox{Nat}_8(x)$ : $E_x$ is axiom 8 of number theory.
\begin{eqnarray*}
&& \hskip-70pt  (\exists a < x )(\exists b < x )
(\mbox{Term}(a) \land \mbox{Term}(b) \land \\
&& \hskip-20pt  x= a \star 2^9\star b \star 2^0\star 2^{10} \star a \star
 2^9\star  b \star 2^{8}\star a)
\end{eqnarray*}
\end{enumerate}
\end{description}

That $E_x$ is the axiom of mathematical induction is expressed as follows.
\begin{description}
\item[25.]  $\mbox{sub}_a(x,y)$ is the G\"odel number of the formula $\forall E_x((E_x$ $= E_y) \Rightarrow (E_a))$ meaning the formal substitution of $E_y$ into the variable $E_x$ of $E_a$.
\begin{eqnarray*}
& \hskip-40pt & 2^{15} \star x \star 2^2\star2^2 \star x  \star 2^{10} \star y  \star
2^3\star2^{11}\star2^2 \star a \star 2^3\star2^3
\end{eqnarray*}
\item[26.] $\mbox{MI}(x)$ : $E_x$ is the axiom of mathematical induction.
\normalsize
\begin{eqnarray*}
&& \hskip-40pt  (\exists a < x ) (\exists b < x )(\exists c < x )
\bigl(\mbox{Form}(a) \land \mbox{Var}(b)  \land \mbox{Var}(c)  \land  \\
&& \hskip-20pt 
x=2^2 \star \mbox{sub}_a(b,2^1) \star
2^{12}\star 2^{15}
\star c \star\ 2^2 \star\mbox{sub}_a(b,c) \star 2^{11} \star \mbox{sub}_a(b,c\star 2^0)\\
&& \hskip-20pt 
\star 2^3\star 2^3\star 2^{11}\star\ 2^{15}\star c \star
 \mbox{sub}_a(b,c)\bigr)
\end{eqnarray*}
\normalsize
\end{description}

We have expressed all axioms of $S$ in the formal system $S$.

\Large

\subsection{Numeralwise expression of being a proof sequence}\label{9.5}


\normalsize

From the above, we can express in the formal system $S$ the fact that a sequence of expressions is a proof sequence consisting of axioms and the results of the applications of rules of inference as follows.
\begin{description}
\item[27.] $\mbox{Axiom}(x)$ : $E_x$ is an axiom.
\begin{eqnarray*}
&\hskip-40pt&\mbox{Pro}(x)\lor \mbox{Pred}_1(x)\lor \mbox{Pred}_2(x)\lor \mbox{Nat}(x)\lor \mbox{MI}(x)
\end{eqnarray*}
\item[28.]  $\mbox{Proof}(x)$ : $E_x$ is a proof sequence.
\begin{eqnarray*}
& \hskip-40pt &\mbox{Seq}(x) \land
\forall y \biggl( y \in x  \Rightarrow \bigl( \mbox{Axiom}(y) \lor (\exists v \prec_x y )  (\exists w \prec_x y ) 
\bigl\{(w = v\star 2^{11} \star y )  \lor \\
& \hskip-40pt &
 (\exists a < v )  (\exists b < v )  (\exists c < y ) \bigl[ v = b \star 2^{11} \star a \land
 y = 
 b \star 2^{11} \star c \land 
\mbox{Gen}(a,c)\land \\
& \hskip-40pt & (\forall z\le a)(\mbox{Var}(z) \Rightarrow \neg \mbox{Part}(z,b))\bigr]\bigr\}
 \bigr)\biggr)
\end{eqnarray*}
\item[29.]  $\mbox{Pr}(x)$ : $E_x$ is provable.
\begin{eqnarray*}
& \hskip-60pt &\exists y \left( \mbox{Proof}(y) \land (x\in y ) \right)
\end{eqnarray*}
\item[30.]  $\mbox{Re}(x)$ : $E_x$ is refutable.
\begin{eqnarray*}
& \hskip-60pt &\exists y \left( 
\mbox{Proof}( y )\land (2^{14} \star 2^2 \star x \star 2^3\in y ) \right)
\end{eqnarray*}
\end{description}

The predicates in {\bf 29} and {\bf 30} are not recursive predicates in the above. Other predicates from {\bf 1} to {\bf 28} are recursive because the latter predicates are all determined to be true or not by making checkings for natural numbers in a finite set. In {\bf 29}, {\bf 30}, there is no restriction to a finite set of natural numbers. Thus in a finitary method one can not determine whether the predicates $\mbox{Pr}(x)$ and $\mbox{Re}(x)$ are true or not.


In the next section, using those descriptions we will show that the predicates $\GGG(a,b)$, $\HHH(a,b)$ are numeralwise expressible.


\section{G\"odel predicate}\label{chap:10}

\normalsize


In this section, using the predicates constructed in the previous section, we will see that the predicates $\GGG(a,b)$, $\HHH(a,b)$ are numeralwise expressed by some formulae $g(a,b)$, $h(a,b)$ in the formal number theory $S$.

\Large

\subsection{Numeralwise expression of G\"odel predicate}\label{10.1}


\normalsize

As stated in section \ref{chap:7}, the predicates $\GGG(a,b)$, $\HHH(a,b)$ which are used in constructing Rosser formula are defined as follows. In this paper we call the predicate $\GGG(a,b)$ G\"odel predicate.

\begin{df}\label{df10.1} {\rm
\begin{enumerate}
\item[ 1)] The predicate $\GGG(a,b)$ means the following.

``A formula $A_a$ with G\"odel number $a$ has just one free variable $x$, and an expression $E_b$ with G\"odel number $b$ is a proof of the formula $A_a=A_a(\aaa)$ obtained from $A_a=A_a(x)$ by substituting $\aaa$ into $x$."
\item[ 2)] The predicate $\HHH(a,b)$ means the following.

``A formula $A_a$ with G\"odel number $a$ has just one free variable $x$, and an expression $E_b$ with G\"odel number $b$ is a proof of the formula $\neg A_a=\neg A_a(\aaa)$ obtained from $\neg A_a=\neg A_a(x)$ by substituting $\aaa$ into $x$."
\end{enumerate}}
\end{df}

The predicates $\GGG(a,b)$, $\HHH(a,b)$ are numeralwise expressible by formulae $g(a,b)$, $h(a,b)$ respectively if the following holds.
\begin{enumerate}
\item[ 1)] 
\begin{enumerate}
\item[ i)] If $\GGG(a,b)$ is true, then $\vdash g(\aaa,\bbb)$ holds.
\item[ ii)] If $\GGG(a,b)$ is false, then $\vdash \neg\hskip1pt g(\aaa,\bbb)$ holds.
\end{enumerate}
\item[ 2)] 
\begin{enumerate}
\item[ i)] If $\HHH(a,b)$ is true, then $\vdash h(\aaa,\bbb)$ holds.
\item[ ii)] If $\HHH(a,b)$ is false, then $\vdash \neg\hskip1pt h(\aaa,\bbb)$ holds.
\end{enumerate}
\end{enumerate}


From the procedures {\bf 1} -- {\bf 28} stated in the previous section follows the next theorem.

\begin{thm}\label{th10.2}{\sl By the G\"odel numbering we have defined for $S$, the predicates $\GGG(a,b)$, $\HHH(a,b)$ in definition \ref{df10.1} are numeralwise expressed by some formulae $g(a,b)$, $h(a,b)$ in $S$.}
\end{thm}
\begin{proof} By definition \ref{df10.1}, these predicates $\GGG(a,b)$, $\HHH(a,b)$ contain a diagonal formula like $A_a(\aaa)$ as in the phrase of $\GGG(a,b)$: ``an expression $E_b$ with G\"odel number $b$ is a proof of the formula $A_a=A_a(\aaa)$ obtained from $A_a=A_a(x)$ by substituting $\aaa$ into $x$," which results from a substitution of itself into itself through G\"odel numbering. To treat these we need, as we will see, to show that the formula $y=2^x$ is number-theoretic for an expression $E_x$ and a natural number $y$. To do so, we have only to show that the pair $(x,y)$ of natural numbers $x,y$ in $y=2^x$ is included in a concrete calculation sequence $(0,1)$, $(1,2)$, $(2,4)$, $(3,8)$, $(4,16)$, $\dots$. However, as $E_x$ can be an expression which includes a sequence of expressions, we cannot use commas in the construction of this calculation sequence. Therefore we need to use numbers $s,t$ which are not G\"odel numbers of any terms, formulae, sequences in the formal system $S$, and need to express the number corresponding to the calculation sequence like $s$ $\star$ $0$ $\star$  $t$ $\star$  $2^0$ $\star$ $s$ $\star$ $2^0$  $\star$ $t$ $\star$ $2^1$  $\star$ $s$ $\star$ $2^1$ $\star$  $t$ $\star$ $2^2$ $\star$  $s$ $\star$  $2^1+2^0$ $\star$  $t$ $\star$  $2^3$ $\star$  $s$ $\star$ $2^2$ $\star$ $t$ $\star$ $2^4$ $\star$ $\dots$. Here we use the assignments $s=2^{18}$, $t=2^{19}$, and will express $y=2^x$ as follows.
\begin{description}
\item[31.] $\mbox{SEQ}(x,y,w)$ : $w$ is a sequence of pairs $(n,m)$ of natural numbers, which includes the pair $(x,y)$.
\begin{eqnarray*}
&& \hskip-40pt \mbox{Part}(2^{18}\star x \star 2^{19}\star y \star 2^{18}, w)\land  \neg \mbox{Part}(2^{18}, x)\land\\
&& \hskip80pt  \neg \mbox{Part}(2^{18}, y)\land \neg \mbox{Part}(2^{19}, x)\land \neg \mbox{Part}(2^{19}, y)
\end{eqnarray*}
\item[32.]  $y=2^x$ : For an expression $E_x$ and a natural number $y$, $y=2^x$ holds.
\begin{eqnarray*}
&& \hskip-40pt  \exists w \biggl(\mbox{SEQ}(x,y,w)
\land (\forall a \leq w) (\forall b \leq w)  \bigl[\mbox{SEQ}(a,b,w)
 \hskip-2pt\Rightarrow\hskip-2pt\\
&& \hskip-40pt
\bigl\{(a=0 \land b=1)
\lor (\exists c\leq a)  (\exists d\leq b) 
[\mbox{SEQ}(c,d,w) \land (a=c+1) \land (b=d\cdot 2)]\bigr\}\bigr]\biggr)
\end{eqnarray*}
\end{description}

Here we need lemma \ref{lem7.5} stated in section \ref{chap:7}, and we restate it as follows.

\begin{lem}\label{lem10.3}{\sl Let $a\ge 0$ be a natural number and let $w$ be the natural number such that $w'=2^a$, where $w'$ is a successor of the natural number $w$ (i.e. $w'=w+1$). Then we have
\beq
g(a)=2^1\star w\label{2-sin-2}.
\ene}
\end{lem}

By this lemma, for a natural number $w$ such that $w'=2^a$, $2^1\star w$ gives the G\"odel number $g(a)$ of the numeral $0^{\prime \prime ...\prime }$ (the number of primes being $a$) corresponding to the natural number $a$. Therefore, the predicates $\GGG(a,b)$ and $\HHH(a,b)$ are numeralwise expressed in the formal system $S$ by the following formulae $g(a,b)$ and $h(a,b)$ respectively.
\begin{description}
\item[33.] $g(a,b)$ : $E_a$ has a free variable $E_x$, and $E_b$ is a proof of the formula $E_a$ when $E_x=a$.
\normalsize
\begin{eqnarray*}
& \hskip-40pt &\exists x \bigl(\mbox{Var}(x) \land \mbox{Part}(x,a) \land \mbox{Free}(x,a) \land \mbox{Proof}(b) \land\\
& \hskip-40pt & \exists w
[w^\prime =2^a \land (\mbox{sub}_a(x,2^1\star w) \in b)]\bigr)
\end{eqnarray*}
\normalsize
\item[34.] $h(a,b)$ : $E_a$ has a free variable $E_x$, and $E_b$ is a proof of $\neg E_a$ when $E_x=a$.
\normalsize
\begin{eqnarray*}
& \hskip-40pt &\exists x \bigl(\mbox{Var}(x) \land \mbox{Part}(x,a) \land \mbox{Free}(x,a) \land \mbox{Proof}(b) \land\\
& \hskip-40pt & \exists w
[w^\prime =2^a \land (2^{14} \star 2^2 \star \mbox{sub}_{a}(x,2^1\star w) \star 2^3\in b )] \bigr)
\end{eqnarray*}
\normalsize
\end{description}
\end{proof}



\Large

\subsection{G\"odel's incompleteness theorem}\label{10.2}


\normalsize

The Rosser formula $A_q(\qqqq)$ is defined as follows.

\begin{df}\label{df10.4}{\rm
Let $q$ be the G\"odel number of the following formula.
$$
\hskip-40pt\forall b  \left( 
g(a,b)\Rightarrow \exists c\hskip1pt(c\le b \hskip1pt\land \hskip0pt h(a,c))
\right) .
$$
Namely
\beq
A_{q}(a) = \forall b  \left(g(a,b)\Rightarrow  \exists c\hskip1pt(c\le b \hskip1pt\land \hskip0pt h(a,c))
\right). \nonumber
\ene
Then the following  formula is called Rosser formula.
\beqs
A_{q}(\qqqq)= \forall b  \left(g(\qqqq,b)\Rightarrow \exists c\hskip1pt(c\le b \hskip1pt\land \hskip0pt h(\qqqq,c))
\right).
\enes
Here
\beqs
&&\hskip-28ptg(\qqqq,b)=\forall a \left(a=q\Rightarrow g(a,b)\right),\\
&&\hskip-28pth(\qqqq,c)=\forall a \left(a=q\Rightarrow h(a,c)\right).
\enes}
\end{df}

The results in the previous subsection have shown theorem \ref{th7.7} in section \ref{chap:7}. Therefore we have the incompleteness theorem of Rosser type by theorem \ref{th7.9} in section \ref{chap:7}.
\SP

On the other hand, G\"odel's original result is as follows.

\begin{df}\label{df10.5}{\rm
Let $p$ be the G\"odel number of the following formula.
$$
\hskip-40pt\forall b \hskip2pt \neg \hskip2pt g(a,b).
$$
Namely
\beqs
A_{p}(a) = \forall b \hskip2pt \neg \hskip2pt g(a,b).
\enes
Then we call the following formula G\"odel sentence or formula.
\beq
A_{p}(\ppp)=\forall b\hskip2pt\neg\hskip2pt g(\ppp,b).\label{godel-bun}
\ene
where
\beqs
&&\hskip-28ptg(\ppp,b)=\forall a \left(a=p\Rightarrow g(a,b)\right).
\enes}
\end{df}

\begin{df}\label{df10.6}{\rm A formal system $S$ which includes the number theory is called $\omega$-consistent if for any variable $x$ and any formula $A(x)$, not all of
$$
A(0),\ A(1),\ A(2),\ \dots\quad\mbox{and}\quad \neg\forall x A(x)
$$
is provable. In particular, if $S$ is $\omega$-consistent, it is (simply) consistent.}
\end{df}

\begin{thm}\label{th10.7}{\sl (G\"odel's incompleteness theorem (1931))  If  the number theory $S$ is consistent, then
$$
\mbox{{\rm not}}\ \vdash A_p(\ppp).
$$
If $S$ is $\omega$-consistent,
$$
\mbox{{\rm not}}\ \vdash \neg A_p(\ppp).
$$
In particular if $S$ is $\omega$-consistent, then $A_p(\ppp)$ is neither provable nor refutable in $S$.}
\end{thm}
\begin{proof}
We assume that $S$ is consistent. Suppose that
\beq
\vdash A_p(\ppp)\label{goedel}
\ene
holds. Then there is a proof of $A_p(\ppp)$. Thus, if we let $b$ be the G\"odel number of the proof, $\GGG(p,b)$ is true. Therefore the numeralwise expressibility of the predicate $\GGG(a,b)$ implies
$$
\vdash g(\ppp,\bbb).
$$
From this and axiom 3 of predicate calculus we have
$$
\vdash \exists b g(\ppp,b).
$$
Namely
$$
\vdash \neg\forall b\neg g(\ppp,b).
$$
This means the following by the definition \eq{godel-bun} of G\"odel formula.
$$
\vdash \neg A_{p}(\ppp).
$$
This and \eq{goedel} show that $S$ is inconsistent. As we have made a premise that $S$ is consistent, \eq{goedel} is wrong. The former part is proved.
\MP

We assume that $S$ is $\omega$-consistent. In particular, $S$ is consistent. Thus by the result above, $A_p(\ppp)$ is not provable in $S$. Thus every natural number $0, 1, 2, \dots$ is not a G\"odel number of a proof of $A_p(\ppp)$. Namely, $\GGG(p,0)$, $\GGG(p,1)$, $\GGG(p,2)$, $\dots$ are all wrong. Therefore by the numeralwise expressibility of the predicate $\GGG(a,b)$, all of
$$
\vdash \neg g(\ppp,\ooo),\ \vdash \neg g(\ppp,\llll),\ \vdash \neg g(\ppp,\tttt),\ \dots
$$
hold. As we assume that $S$ is $\omega$-consistent, from this follows
$$
\mbox{not }\vdash \neg \forall b\neg g(\ppp,b).
$$
By the definition \eq{godel-bun} of G\"odel formula, this means
$$
\mbox{not }\vdash \neg A_{p}(\ppp).
$$
This proves the latter part of the theorem.
\end{proof}


\Large

\subsection{The second incompleteness theorem}\label{10.3}


\normalsize

The former part of G\"odel's theorem \ref{th10.7} is summarized as follows.
\beq
\mbox{\ $S$ is consistent }\Rightarrow\mbox{ $A_p(\ppp)$ is not provable.}\label{goedel-matome}
\ene
By \eq{godel-bun}, the fact that ``$A_p(\ppp)$ is not provable" on the right hand side is written by translating it by G\"odel numbering as follows:
\beq
A_{p}(\ppp)=\forall b\hskip2pt\neg\hskip2pt g(\ppp,b).\label{godel-bun-2}
\ene
Therefore if we translate and map the metamathematics by G\"odel numbering into $S$ and write the fact that ``$S$ is consistent" by a formal formula in $S$:
$$
\mbox{Consis}(S),
$$
then together with the formula \eq{godel-bun-2} we have from the first part of theorem \ref{th10.7} that
\beq
\vdash \mbox{Consis}(S)\Rightarrow A_{p}(\ppp)\label{dai-2-goedel}.
\ene

Now let us assume on the meta level
$$
\vdash \mbox{Consis}(S).
$$
Then together with the formula \eq{dai-2-goedel}, we have
$$
\vdash A_{p}(\ppp).
$$
This contradicts the first part of theorem \ref{th10.7}. Therefore we have the following theorem.

\begin{thm}\label{th10.8}{\sl (G\"odel's second incompleteness theorem (1931)) If the number theory $S$ is consistent, then
$$
\mbox{{\rm not} }\vdash \mbox{{\rm Consis}}(S)
$$
holds. Namely if $S$ is consistent, the consistency of $S$ is not proved by a method formalizable in $S$.}
\end{thm}

The proof above is an outline. A complete proof is given in Hilbert-Bernays (1939). We note that this theorem is proved without using Rosser's stronger result: theorem \ref{th7.9}.

\subsection{Implications of the second theorem}\label{10.4}

As stated at the beginning of section \ref{chap:2} and in subsection \ref{2.1}, the G\"odel's second incompleteness theorem is thought to have shown the impossibility of the Hilbert's formalism or program ``if one could show the consistency of the formal axiomatic system of classical mathematics from the finitary standpoint, it shows the soundness of the formal system treating infinities." This is because the second theorem is interpreted as meaning ``if a system $S$ is consistent, it is impossible to show the consistency of $S$ by a method having the power equivalent to that of $S$," and because the finitary method is thought equal to the ability of number theory $S$.

If the first and hence the second theorem is proved by a completely syntactic method, this interpretation is surely true. However as we have seen above, already in the proof of the first incompleteness theorem, a semantic interpretation is assumed as suggested at the end of subsection \ref{1.3}. Namely as mentioned at the beginning of section \ref{chap:7}, in the process of replacing the natural number $n$ by a numeral $\nnn$ of the formal system, a substitution is made with assuming an identification of the meta leveled theory and the object leveled theory. This is a self-reference stated in section \ref{chap:1}. As stated in subsection \ref{1.3}, if self-reference is not made, one would have almost nothing to speak. However, we have also seen in the same subsection \ref{1.3} that if one pursues the self-reference to an ultimate point, he will meet a contradiction. The proof of G\"odel's theorem was made by making a complete self-reference by embedding the discussion on the meta level (which is the subject who does mathematics) into the formal system (which is his own object theory). It is at a glance a syntactic argument, however it assumes the symmetry or reflexivity between the meta level and the object level. To assume the symmetry between the meta level and the object level in the case of number theory means that the discussion is not based merely on the syntactic treatment of words, but the meaning of the object theory of natural numbers is applied to the discussion of the meta level. This reflects the fact that the discussion on the meta level is influenced by the objects of the research activity. Here seems to appear the ordinary phenomenon that professionals or researchers are often influenced by the objects of their own job or study.

Thinking like this, it is not odd that the incomleteness theorem was found by mathematicians and computation scientists who study numbers. For logicians in the medieval age, their object of study would have been classical documents, so it is not surprising that they did not try to think by replacing things by numbers.

Probably behind the incompleteness theorem are hidden the things like these. It is then natural that the incompleteness theorem demands the validity of the theorem of the present or modern people who have no or almost no sense of incongruity with the thought of replacing their own thinking by numbers.

How about Cretan paradox or other self-referential paradoxes or the following Tarski's theorem\footnote{cf. section 6.1 of \cite{Kitada-book}.} arising from the interpretation of those paradoxes then? 
\begin{quote}
The set of true sentences of a language ${\cal L}$ is not referred to by a sentence inside the language ${\cal L}$. Namely the predicate $T$ showing the truthness must not be inside the language ${\cal L}$.
\end{quote}
These are about the ordinary language and are not related with numbers.

Probably this is also the problem in the same category. In either case of numbers or language one can do reflection and rumination only when symbolization of them has been made, and no problem as above arises without the ``objectification of oneself." The problem arises only when the subject of thinking on the meta level objectifies himself and makes himself the object of thinking. It is the conveniences produced by civilization, the letters, which are working here. The letters are indispensable tools for humans in forming the future plans as well as are troublesome tools which allow them self-reflections. In any case, it is necessary to make symbolization of language for humans' life, but there inevitably accompanies the self-reflections. Then necessarily one is always in the inconsistent world which G\"odel's incompleteness theorem presents. Namely without language or its symbolization, the human as a subject of thinking would have been free from the influence of the object world of the thinking. Once a symbolization of language however is made, the self-reflection occurs inevitably, and humans have come to carry the undecidable self-referential problems always.
\BP

Until now we have been reviewing the proof of G\"odel's incompleteness theorem. We now turn to a problem which arises accompanying the incompleteness theorem: What is consequent if one continues to add the undecidable proposition whose existence is proved by the incompleteness theorem, to the axioms of $S$? For example if one adds the proposition $A_q(\qqqq)=A_{q^{(0)}}(\qqqqo)$ which is undecidable in $S$ by theorem \ref{th7.9} to the axioms of $S$, one can repeat the argument of the incompleteness theorem and gets a new undecidable proposition $A_{q^{(1)}}(\qqqqa)$. In the same way one can produce infinite similar undecidable propositions $A_{q^{(n)}}(\qqqqn)$ $(n=0,1,2,\dots)$.  If then one adds all of these infinite propositions to the axioms of $S$, can he repeat the argument of G\"odel's incompleteness theorem? If he can, how about repeating the argument transfinitely infinite times? Right from the very beginning, the formal system we have been thinking in this article is able to express propositions of at most countable infinity. To such a system if we are able to add the uncountably infinite propositions as axioms, the system $S$ must have propositions of more than countable infinity. In metamathematics it seems to be generally assumed\footnote{cf. a later subsection \ref{11.5}.} that such a transfinite procedure is possible only up to the extent that the added axioms are countable. However insofar as we have seen, there seems to be no reason for it. If one can then make transfinite operations on the meta level, this is certainly inconsistent with the countability of the number of propositions of the system $S$.
\SP

In the following sections we will consider those problems.

\section{Mathematics is inconsistent?}\label{chap:11}

In this section we will consider the problem whether the set theory ZFC which is thought to be a basis of modern mathematics is consistent. As stated in section \ref{chap:2}, the discovery of Russell's set in 1903 gave an impact to the foundation of mathematics and became a source of contention. To deal with the situation, there were proposed Hilbert's formalism, and other standpoints. As stated the Hilbert's standpoint that ``if one can show the consistency of mathematics based on finitary standpoint, mathematics is sound" seems to have been answered negatively by the G\"odel's incompleteness theorem. In this section we will see that the G\"odel's theorem seems to imply that mathematics is contradictory.


\subsection{Incompleteness theorem of Rosser type, revisited}\label{11.1}

\normalsize


We now consider a formal set theory $S$ equivalent to ZFC, and assume that we can use the same set theory ZFC also on the meta level. We can develop a number theory in this formal system $S$. We denote this subsystem of number theory by $S^{(0)}$. Then the G\"odel predicate $\GGG^{(0)}(a,b)$ and the related predicate $\HHH^{(0)}(a,b)$ are defined as follows.

\begin{df}\label{df11.1} {\rm
\begin{enumerate}
\item[ 1)] The predicate $\GGG^{(0)}(a,b)$ means the following.

``A formula $A_a$ with G\"odel number $a$ has just one free variable $x$, and an expression $E_b$ with G\"odel number $b$ is a proof of the formula $A_a=A_a(\aaa)$ obtained from $A_a=A_a(x)$ by substituting $\aaa$ into $x$."
\item[ 2)] The predicate $\HHH^{(0)}(a,b)$ means the following.

``A formula $A_a$ with G\"odel number $a$ has just one free variable $x$, and an expression $E_b$ with G\"odel number $b$ is a proof of the formula $\neg A_a=\neg A_a(\aaa)$ obtained from $\neg A_a=\neg A_a(x)$ by substituting $\aaa$ into $x$."
\end{enumerate}}
\end{df}

For these predicates we have shown the following in the former sections.

\begin{thm}\label{th11.2}{\sl
By the G\"odel numbering we have stated before, the predicates $\GGG^{(0)}(a,b)$, $\HHH^{(0)}(a,b)$ in definition \ref{df11.1} are numeralwise expressed by corresponding formulae $g^{(0)}(a,b)$, $h^{(0)}(a,b)$ in $S^{(0)}$, therefore in the formal set theory $S$. Namely the following holds. Let the formulae $g^{(0)}(a,b)$ and $h^{(0)}(a,b)$ be defined as follows.
\begin{enumerate}
\item[ 1)] $g^{(0)}(a,b)$ : $E_a$ has a free variable $E_x$, and $E_b$ is a proof of the formula $E_a$ when $E_x=a$.
\normalsize
\begin{eqnarray*}
& \hskip-40pt &\exists x \bigl(\mbox{{\rm Var}}(x) \land \mbox{{\rm Part}}(x,a) \land \mbox{{\rm Proof}}(b) \land\\
& \hskip-40pt & \exists w
[w^\prime =2^a \land (\mbox{{\rm sub}}_a(x,2^1\star w) \in b)]\bigr)
\end{eqnarray*}
\normalsize
\item[ 2)] $h^{(0)}(a,b)$ : $E_a$ has a free variable $E_x$, and $E_b$ is a proof of $\neg E_a$ when $E_x=a$.
\normalsize
\begin{eqnarray*}
& \hskip-40pt &\exists x \bigl(\mbox{{\rm Var}}(x) \land \mbox{{\rm Part}}(x,a) \land \mbox{{\rm Proof}}(b) \land\\
& \hskip-40pt & \exists w
[w^\prime =2^a \land (2^{14} \star 2^2 \star \mbox{{\rm sub}}_{a}(x,2^1\star w) \star 2^3\in b )] \bigr)
\end{eqnarray*}
\normalsize
\end{enumerate}
Then the following holds.
\begin{enumerate}
\item[ (1)] 
\begin{enumerate}
\item[ i)] If $\GGG^{(0)}(a,b)$ is true, then $\vdash g^{(0)}(\aaa,\bbb)$ holds.\item[ ii)] If $\GGG^{(0)}(a,b)$ is false, then $\vdash \neg\hskip1pt g^{(0)}(\aaa,\bbb)$ holds.
\end{enumerate}
\item[ (2)] 
\begin{enumerate}
\item[ i)] If $\HHH^{(0)}(a,b)$ is true, then $\vdash h^{(0)}(\aaa,\bbb)$ holds.
\item[ ii)] If $\HHH^{(0)}(a,b)$ is false, then $\vdash \neg\hskip1pt h^{(0)}(\aaa,\bbb)$ holds.
\end{enumerate}
\end{enumerate}}
\end{thm}

\begin{df}\label{df11.3}{\rm Let $q^{(0)}$ be the G\"odel number of the formula
$$
\forall b [\neg g^{(0)}(a,b)\vee \exists c(c\le b \hskip3pt\wedge\hskip2pt h^{(0)}(a,c))].
$$
Namely 
\beq
A_{q^{(0)}}(a)=\forall b [\neg g^{(0)}(a,b)\vee \exists c(c\le b \hskip3pt\wedge\hskip2pt h^{(0)}(a,c))].\nonumber
\ene
We then define Rosser formula in $S^{(0)}$ as follows.
\beqs
&&\hskip-20ptA_{q^{(0)}}(\qqqqo)=\forall b [\neg g^{(0)}(\qqqqo,b)\vee \exists c(c\le b \hskip3pt\wedge\hskip2pt h^{(0)}(\qqqqo,c))].
\enes}
\end{df}

Then the incompleteness theorem of Rosser type for $S^{(0)}$ is as follows.

\begin{lem}\label{lem11.4}{\sl If $S^{(0)}$ is consistent, both of $A_{q^{(0)}}(\qqqqo)$ and $\neg A_{q^{(0)}}(\qqqqo)$ are unprovable in $S^{(0)}$.}
\end{lem}

The proof is given in theorem \ref{th7.9}.

\subsection{Extension of $S^{(0)}$}\label{11.2}

\normalsize

By lemma \ref{lem11.4}, if we let either of $A_{q^{(0)}}(\qqqqo)$ or $\neg A_{q^{(0)}}(\qqqqo)$ be $A_{(0)}$, and add $A_{(0)}$ to the axioms of $S^{(0)}$ as a new axiom to obtain a new system $S^{(1)}$, we have
\beq
S^{(1)} \mbox{ is consistent.}\label{consis}
\ene

We extend the G\"odel numbering for $S^{(0)}$ in theorem \ref{th11.2} to the system $S^{(1)}$, and extend definitions \ref{df11.1} and \ref{df11.3} to the system $S^{(1)}$ as follows.
\begin{enumerate}
\item[ 1)] The predicate $\GGG^{(1)}(a,b)$ means the following.

``A formula $A_a$ with G\"odel number $a$ has just one free variable $x$, and an expression $E_b$ with G\"odel number $b$ is a proof of the formula $A_a=A_a(\aaa)$ obtained from $A_a=A_a(x)$ by substituting $\aaa$ into $x$."

\item[ 2)] The predicate $\HHH^{(1)}(a,b)$ means the following.

``A formula $A_a$ with G\"odel number $a$ has just one free variable $x$, and an expression $E_b$ with G\"odel number $b$ is a proof of the formula $\neg A_a=\neg A_a(\aaa)$ obtained from $\neg A_a=\neg A_a(x)$ by substituting $\aaa$ into $x$."
\end{enumerate}

In the same way as before we can show that the predicates $\mbox{\GGG}^{(1)}(a,b)$ and $\mbox{\HHH}^{(1)}(a,c)$ are numeralwise expressed by the corresponding formulae $g^{(1)}(a,b)$ and $h^{(1)}(a,c)$ in $S$ respectively.

\begin{enumerate}
\item[ 3)]
Let $q^{(1)}$ be the G\"odel number of the formula
$$
\forall b [\neg g^{(1)}(a,b)\vee \exists c(c\le b \hskip3pt\wedge\hskip2pt h^{(1)}(a,c))].
$$
Namely
\beqs
&&\hskip-22ptA_{q^{(1)}}(a)=\forall b [\neg g^{(1)}(a,b)\vee \exists c(c\le b \hskip3pt\wedge\hskip2pt h^{(1)}(a,c))].\nonumber
\enes
Then
\beqs
&&\hskip-22ptA_{q^{(1)}}(\qqqqa)=\forall b [\neg g^{(1)}(\qqqqa,b)\vee \exists c(c\le b \hskip3pt\wedge\hskip2pt h^{(1)}(\qqqqa,c))].\nonumber
\enes
\end{enumerate}

Using the numeralwise expressibility of the predicates $\mbox{\GGG}^{(1)}(a,b)$ and $\mbox{\HHH}^{(1)}(a,c)$ and the consistency of $S^{(1)}$ in \eq{consis}, we can show in the same way as in the proof of lemma \ref{lem11.4}
\beq
&&\hskip-15pt\mbox{not } \vdash A_{q^{(1)}}(\qqqqa) \mbox {\ \ and \ }\mbox{not } \vdash \neg A_{q^{(1)}}(\qqqqa) \mbox{ in } S^{(1)}\label{pos1}.\nonumber
\ene

\MP

Then we let either of $A_{q^{(1)}}(\qqqqa)$ or $\neg A_{q^{(1)}}(\qqqqa)$ be $A_{(1)}$, and can add $A_{(1)}$ as a new axiom of $S^{(1)}$ to obtain a new system $S^{(2)}$. Then $S^{(2)}$ is consistent.
\MP

Proceeding in the same way, we have for any natural number $n(\ge 0)$
\beq
S^{(n)} \mbox{ is consistent}\label{consis2}
\ene
and
\beq
&&\hskip-20pt\mbox{not } \vdash A_{q^{(n)}}(\qqqqn)\mbox{\ \ and \ }\mbox{not } \vdash \neg A_{q^{(n)}}(\qqqqn) \mbox{ in } S^{(n)}.\nonumber\label{posn}
\ene

\subsection{Infinite extension of $S^{(0)}$}\label{11.3}

\normalsize

We denote by $S^{(\omega)}$ the system obtained by adding all of the following as new axioms to the system $S^{(0)}$.
$$
A_{(n)}=A_{q^{(n)}}(\qqqqn)\mbox{ \ \ or\ \ }\ \neg A_{q^{(n)}}(\qqqqn)\ (n\ge0)
$$
Then by \eq{consis2}, the system $S^{(\omega)}$ is consistent. Let ${\widehat q}(n)$ be the G\"odel number of the formula $A_{(n)}$. The formula $A_{(j)}$ is not provable in $S^{(i+1)}$ for $i<j$. Thus if $i<j$, the system $S^{(i)}$ is a proper subsystem of $S^{(j)}$, and ${\widehat q}(i)<{\widehat q}(j)$. Therefore for a given formula $A_r$ with G\"odel number $r$, we can decide\footnote{In fact by the monotonicity of the sequence ${\widehat q}(n)$ and its recursive definition, it is possible to decide whether $A_r$ is the axiom of the form $A_{(n)}$ in a recursive way. cf. e.g. \cite{C}, Chapter 5. However as we assume ZFC on the meta level, we can make this decision by the axioms of ZFC even if we cannot make this decision recursively.} whether $A_r$ is an axiom of the form $A_{(n)}$ by comparing the given formula $A_r$ with a finite number of axioms $A_{(n)}$ with $\widehat{q}(n)\le r$. From this fact we can define the following two predicates on the meta level of $S^{(\omega)}$, if we assume the same G\"odel numbering for the system $S^{(\omega)}$ as the one for the system $S^{(0)}$ in theorem \ref{11.2}.
\begin{enumerate}
\item[ 1)] The predicate $\GGG^{(\omega)}(a,b)$ means the following.

``A formula $A_a$ with G\"odel number $a$ has just one free variable $x$, and an expression $E_b$ with G\"odel number $b$ is a proof of the formula $A_a=A_a(\aaa)$ obtained from $A_a=A_a(x)$ by substituting $\aaa$ into $x$."
\item[ 2)] The predicate $\HHH^{(\omega)}(a,b)$ means the following.

``A formula $A_a$ with G\"odel number $a$ has just one free variable $x$, and an expression $E_b$ with G\"odel number $b$ is a proof of the formula $\neg A_a=\neg A_a(\aaa)$ obtained from $\neg A_a=\neg A_a(x)$ by substituting $\aaa$ into $x$."
\end{enumerate}

The predicates $\GGG^{(\omega)}(a,b)$ and $\HHH^{(\omega)}(a,b)$ are numeralwise expressed\footnote{Note that in the proof of theorem \ref{10.2} we do not use the recursiveness of the predicates $\GGG(a,b)$, $\HHH(a,b)$. Theorem \ref{10.2} thus yields that the set-theoretic predicates $\GGG^{(\omega)}(a,b)$, $\HHH^{(\omega)}(a,b)$ on the meta level are directly expressed by the formal formulae $g^{(\omega)}(a,b)$, $h^{(\omega)}(a,c)$ of system $S$ which is a formalization of set theory ZFC.} by corresponding  formulae $g^{(\omega)}(a,b)$ and $h^{(\omega)}(a,c)$ in $S$.

\begin{enumerate}
\item[ 3)]
Let $q^{(\omega)}$ be the G\"odel number of the formula
$$
\forall b [\neg g^{(\omega)}(a,b)\vee \exists c(c\le b \hskip3pt\wedge\hskip2pt h^{(\omega)}(a,c))].
$$
Namely
\beqs
&&\hskip-22ptA_{q^{(\omega)}}(a)=\forall b [\neg g^{(\omega)}(a,b)\vee \exists c(c\le b \hskip3pt\wedge\hskip2pt h^{(\omega)}(a,c))].\nonumber
\enes
Then
\beqs
&&\hskip-22ptA_{q^{(\omega)}}(\qqqqomega)=\forall b [\neg g^{(\omega)}(\qqqqomega,b)\vee \exists c(c\le b \hskip3pt\wedge\hskip2pt h^{(\omega)}(\qqqqomega,c))].\nonumber
\enes
\end{enumerate}

From these and the consistency of $S^{(\omega)}$, similarly to lemma \ref{lem11.4}, we obtain
\beq
&&\hskip-20pt\mbox{not } \vdash A_{q^{(\omega)}}(\qqqqomega) \mbox {\ \ and \ \  }\mbox{not } \vdash \neg A_{q^{(\omega)}}(\qqqqomega) \mbox{ in } S^{(\omega)}.\label{posomega}\nonumber
\ene

\subsection{Transfinite extension of $S^{(0)}$}\label{11.4}

\normalsize

Now we let
$$
A_{(\omega)}=A_{q^{(\omega)}}(\qqqqomega)\mbox{\ \ or\ \ }\neg A_{q^{(\omega)}}(\qqqqomega)
$$
and we add this as an axiom to the system $S^{(\omega)}$ to obtain a new system $S^{(\omega+1)}$. Then in a similar way as above we obtain
\begin{quotation}
If $S^{(0)}$ is consistent, then $S^{(\omega+1)}$ is consistent.
\end{quotation}

Repeating this procedure in a similar way transfinitely, we can construct\footnote{cf. former footnotes 11, 12.} a consistent formal system $S^{(\alpha)}$ for any ordinal number $\alpha$, which is an extension of $S^{(0)}$. Namely we have
\begin{quotation}
If $S^{(0)}$ is consistent, then $S^{(\alpha)}$ is consistent.
\end{quotation}

However if we can construct a formal system $S^{(\alpha)}$ for any ordinal $\alpha$, the number of the totality of axioms $A_{(\alpha)}$ added at each step will be greater than countable infinity\footnote{This is because we assume the axiom of choice. cf. section 8.3, lines 6-9 of page 199 and theorem 8.10 in section 8.2 of \cite{Kitada-book}.}. The number of the formulae of the system $S$ is at most countable as each formula consists of a finite number of primitive symbols. This is a contradiction. Thus the extension like this must stop at a countable\footnote{This is invariantly true even if the system $S$ has primitive symbols more than countable infinity. In this case certainly the number of formulae of $S$ can be greater than countable infinity. However the axioms $A_{(\alpha)}$ that can be added as axioms are undecidable propositions of the system $S^{(\alpha)}$, and for $A_{(\alpha)}$ to be such a proposition, $\alpha$ must be countable. Namely for G\"odel predicate $\GGG^{(\alpha)}(a,b)$ to be numeralwise expressible in $S$, it is necessary that $\alpha$ is at most countable.} ordinal $\beta_0$. Namely we have shown the following.

\begin{thm}\label{th11.5}{\sl There is a countable limit ordinal $\beta_0$ such that when $\alpha=\beta_0$, the system $S^{(\alpha)}$ has no undecidable proposition, and $S^{(\beta_0)}$ is complete. In other words, any extension of $S^{(\beta_0)}$ is inconsistent.}
\end{thm}
\begin{proof} We have only to show that $\beta_0$ is a limit ordinal. In fact if $\beta_0=\delta+1$, $S^{(\beta_0)}$ is obtained by adding the axiom $A_{(\delta)}$ to $S^{(\delta)}$. In this case, by the same method mentioned above, the system $S^{(\beta_0)}=S^{(\delta+1)}$ can be extended with retaining consistency, contradicting the fact that the extension ends at $\beta_0$.
\end{proof}

\subsection{Church-Kleene ordinal}\label{11.5}

\normalsize

Feferman \cite{F} considers as an axiom $A_{(\alpha)}$ to be added to the system $S^{(\alpha)}$ a proposition Consis${}_{(\alpha)}$ meaning that ``$S^{(\alpha)}$ is consistent." This proposition is undecidable by G\"odel's second incompleteness theorem. He showed that the extension ends at a countable ordinal $\omega_1=\omega_1^{CK}$ called Church-Kleene ordinal. According to him
$$
\omega_1\\<\omega^{\omega^{\omega^2}}.
$$
In our context $\beta_0=\omega_1$, and by theorem \ref{th11.5}, the extension of our system $S^{(\alpha)}$ ends at $\alpha=\omega_1$. Namely if $S^{(0)}$ is consistent, $S^{(\omega_1)}$ cannot be extended further with retaining consistency, i.e. $S^{(\omega_1)}$ is complete.

The ordinal $\omega_1$ is a countable limit ordinal by theorem \ref{th11.5}. Therefore we can take a monotone increasing sequence $\{\alpha_n\}_{n=0}^\infty$ of countable ordinals such that $\alpha_n<\omega_1$ $(n=0,1,2,\dots)$ and
$$
\omega_1=\bigcup_{n=0}^\infty \alpha_n.
$$
The axioms $A_{(\gamma)}$ $(\gamma<\omega_1)$ of $S^{(\omega_1)}$ are the sum of the axioms $A_{(\gamma)}$ $(\gamma<\alpha_n)$ of $S^{(\alpha_n)}$. By the definition of ${\widehat q}(\gamma)$ for $\gamma<\alpha_n$, it is possible to decide whether a given formula $A_r$ is an axiom of $S^{(\alpha_n)}$ by seeing whether $A_{(\gamma)}=A_r$ for a finite number of $\gamma$ with ${\widehat q}(\gamma)\le r$. Therefore to see if a given formula $A_r$ is an axiom of $S^{(\omega_1)}$, it is sufficient to see if $A_{(\gamma)}=A_r$ for a finite number of $\gamma$ such that ${\widehat q}(\gamma)\le r$, $\gamma<\omega_1$. By
$$
\omega_1=\bigcup_{n=0}^\infty \alpha_n,
$$
we have
$$
{\widehat q}(\gamma)\le r\wedge\gamma<\omega_1\Leftrightarrow \exists n\ \left[{\widehat q}(\gamma)\le r\wedge\gamma<\alpha_n\right].
$$
Thus whether a given formula $A_r$ is an axiom of $S^{(\omega_1)}$ is decided by an induction on $n$.

We define as follows.
\begin{enumerate}
\item[ 1)] The predicate $\GGG^{(\omega_1)}(a,b)$ means the following.

``A formula $A_a$ with G\"odel number $a$ has just one free variable $x$, and an expression $E_b$ with G\"odel number $b$ is a proof of the formula $A_a=A_a(\aaa)$ obtained from $A_a=A_a(x)$ by substituting $\aaa$ into $x$."
\item[ 2)] The predicate $\HHH^{(\omega_1)}(a,b)$ means the following.

``A formula $A_a$ with G\"odel number $a$ has just one free variable $x$, and an expression $E_b$ with G\"odel number $b$ is a proof of the formula $\neg A_a=\neg A_a(\aaa)$ obtained from $\neg A_a=\neg A_a(x)$ by substituting $\aaa$ into $x$."
\end{enumerate}
Then these predicates are numeralwise expressed by the corresponding formulae $g^{(\omega_1)}(a,b)$, $h^{(\omega_1)}(a,c)$ in $S$. Thus the G\"odel number $q^{(\omega_1)}$ of the formula
\beqs
&&\hskip-22ptA_{q^{(\omega_1)}}(a)=\forall b [\neg g^{(\omega_1)}(a,b)\vee \exists c(c\le b \hskip3pt\wedge\hskip2pt h^{(\omega_1)}(a,c))]
\enes
in $S^{(\omega_1)}$ is defined. Therefore the system $S^{(\omega_1)}$ has an undecidable proposition
$$
A_{q^{(\omega_1)}}(\qqqqomegaa)),
$$
and the incomleteness theorem holds for the system $S^{(\omega_1)}$. Hence $S^{(\omega_1)}$ is incomplete. This contradicts the following consequence of theorem \ref{th11.5}.
\begin{quotation}
\F
``The extension of the system $S^{(\alpha)}$ ends at $\alpha=\omega_1$, and if $S^{(0)}$ is consistent, it is impossible to extend $S^{(\omega_1)}$ further with retaining consistency. Namely if $S^{(0)}$ is consistent, then $S^{(\omega_1)}$ is complete."
\end{quotation}


The argument in this subsection is valid even if we do not introduce Church-Kleene ordinal. We can make the same argument for the countable ordinal $\beta_0$ in theorem \ref{th11.5}. Therefore the contradiction stated in this subsection already follows from theorem \ref{th11.5}, and what was stated in this section holds also for ZF.

In this way we meet a contradiction if we assume that set theory holds on the meta level and discuss the set theory as an object theory.

In Hilbert's formalism it is only permitted to do finite procedures on the meta level, and one tries to treat the object theory which deals with the infinities. If we take this standpoint, the contradiction mentioned above does not seem to appear. However, the consistency of the object world implies its incompleteness. Further the consistency itself is not decidable.

However if we reflexively assume that the object world takes also finitary standpoint, then it seems that there may appear no contradiction, and the object world would be complete. In other words, if we do not assume the axiom of infinity in the object set theory, the meta leveled world and the object leveled world are symmetric or reflexive, and on this setting, the object world will be consistent and complete. The cause that a contradiction appeared in the above argument is that we assumed the axiom of infinity on the both levels of the object world and the meta world. Namely the cause is that we assumed that the actual infinity exists in both of the object and meta worlds. If we take the standpoint that mathematical existence is only computable things and that the infinity is not an actual one but is a fictitious existence which is an auxiliary tool for the inquiry of the computability, Hilbert's thesis that consistency and completeness are the certification of the soundness of mathematics will revive.


In the next section we will see the deeper problem hidden behind these.

\section{Self-reference and inconsistency}\label{chap:12}

As we have seen in section \ref{chap:11}, if we apply to the number theory $S^{(0)}$ as a subsystem of set theory $S$ the argument of incompleteness theorem repeatedly and continue to add undecidable propositions to the axioms of $S^{(0)}$, we finally arrive at a countable ordinal $\beta_0$ and the corresponding extended system $S^{(\beta_0)}$ must be a complete system and has no undecidable propositions. However we have seen that it is possible to construct an undecidable Rosser formula in this system $S^{(\beta_0)}$ and the argument of incompleteness theorem is also valid. This is a contradiction.

\subsection{Cause of inconsistency}\label{12.1}

We stated that the cause of the contradiction is that we have assumed that one can use the set theory on the meta level to the same extent as the set theory on the object level. If, returning to the original point of formalism, we assume that only the finitary method is used on the meta level, the extension of $S^{(0)}$ is possible merely by the usual mathematical induction. Hence the extension of $S^{(0)}$ is allowed at most to $S^{(\omega)}$ which is incomplete, and further extension is not allowed by the finite ability of meta level\footnote{In mathematics which stands on the finitary method, we cannot think of an infinite set like $\omega=\{0,1,2,\dots\}$ as a set.}. Therefore on this standpoint there will be no contradiction mentioned above.

Further as stated at the end of subsection \ref{11.5}, if we restrict the mathematics on the object level to the finite mathematics that the intuitionism admits, we can preserve the reflexivity of the meta level and the object level with retaining the consistency of both levels.

What we have stated above is a solution when we assume that there is no problem in the argument of the proof of G\"odel's incompleteness theorem. As mentioned already at several places, the proof of incompleteness theorem is possible by allowing some self-reference or `confusion' of identifying the numbers on the meta level with the numbers on the object level by substitutions. Namely as stated at the beginning of section \ref{chap:7}, we have defined the operation of substitution of the natural number $n$ on the meta level into a formula $F(x)$ with letting $x=\nnn$ by
\begin{eqnarray}
F({\nnn})\stackrel{\scriptsize\mbox{{\it def}}}=\forall x \left(x=n\Rightarrow F\right).\label{dainyuu2}
\end{eqnarray}
We here make an apparent confusion of identifying the meta level with the object level. This is an implicit assumption that corresponds to the Russell's axiom of reducibility in the sense that we have identified the natural number $n$ on the meta level with the numeral $\nnn$ on the object level. In the number theory considered in G\"odel's original paper, the axiom of reducibility is assumed as it uses the system based on Principia Mathematica of Whitehead and Russell. However the `implicit assumption' that was referred to above as corresponding to the axiom of reducibility is not this one that is explicitly referred to as the axiom of reducibility in G\"odel's original paper. We are speaking of the `implicit assumption' that for a given natural number $x_j$ on the meta level, one constructs a numeral $Z(x_j)$ on the object level and substitutes it into the variable $u_j$ of the formula $r(u_1,\dots,u_n)$ on the object level. Namely we are pointing out the confusion made in the operation of forming a numeral $Z(x_j)$ on the object level from a natural number $x_j$ on the meta level.

\subsection{Self-reference and inconsistency}\label{12.2}

\normalsize

As stated in subsection \ref{10.4}, the second incompleteness theorem is thought usually as claiming the impossibility of Hilbert's formalism that mathematics is sound if one can show the consistency of mathematics by writing the classical mathematics which treats the infinity as a formal axiomatic system. If the incompleteness theorem is shown completely in a syntactic manner, it would be true to assert so. However as we have seen, already in the proof of the first incompleteness theorem, one has made a `semantic interpretation' when replacing the natural number $n$ on the meta level by the numeral $\nnn$ on the object level. This is the self-reference that was mentioned in section \ref{chap:1}. And the reflexivity or the symmetricity between the meta level and the object level means this self-reference or the confusion between the two levels.

Therefore the true cause of the contradiction is said to be in allowing the self-reference. However as we have considered in section \ref{chap:1}, speaking is nothing but a self-reference itself for humans.

Then without quitting being humans, how can we avoid the contradiction?

As stated at the end of subsection 7.5, if we allow self-reference, the only way to avoid the contradiction is to speak of finite things only.

Probably for living beings which have no language or its `symbolization,' there would be no problem which we meet. Self-reference is thought to be a problem proper to humans who have language or its symbolization. In humans who have acquired the ability to control things actually in concrete manner, the obtainment of language gave them the power to work on the things and nature. There always associate byproducts with all things. The language gave humans ability and power. However at the same time it gave them the possibility to come in the infinite cycles of self-reference.

\subsection{Restriction of self-reference}\label{12.3}

\normalsize

As seen in section \ref{chap:2}, the paradoxes like Russell's one which looks coming from the self-reference produced the formalism that proposed that if one writes down the mathematics in a formal system of symbols in finitary method, the contradiction would disappear. The avoidance of Russell's paradox was done in this direction by a formal axiomatic set theory. This approach also excluded the Burali-Forti's paradox of the set of all ordinals and Cantor's paradox of the set of all sets by regarding them as not-sets or proper classes\footnote{see for details chapters 7, 8 of \cite{Kitada-book}.}.

There are many problems which arise by self-reference. For example, there is the problem of `impredicative definition.' This refers to the situation that a set $M$ and an object $m$ are defined as follows. Namely on the one hand, $m$ is an element of the set $M$, and on the other hand, the definition of $m$ depends on $M$. Similarly the terminology is used in the case when for a property $P$, an object $m$ whose definition depends on $P$ satisfies the property $P$. In the latter words, the set $M$ above is the set of all elements which satisfy the property $P$. Apparently these situations are `cyclic.' Poincar\'e (1905-6)  asserted that the cause of the paradox like this is the vicious circle of discussions, and Russell made the same opinion as his vicious circle principle (1906) which claims to prohibit such cyclic definitions. This principle can exclude Russell's paradox, the paradox of all sets, etc. However how about the following concrete example of analysis?

The definition of supremum $\sup M$ of a subset $M$ of real numbers is as follows in the Dedekind's construction of real numbers by his notion `cut.' Let $\R$ be the totality of real numbers and let $\Q$ be the totality of rational numbers. An element $\alpha$ of $\R$ is defined as a set of rational numbers which satisfies the following three properties.
\begin{enumerate}
\item $\alpha\ne \emptyset$,$\quad$ $\alpha^c:=\Q-\alpha=\{s|s\in\Q\ \wedge\ s\notin\alpha\}\ne\emptyset.$
\item $r\in\alpha,s<r,s\in\Q\Rightarrow s\in \alpha$.
\item $\alpha$ has no maximum element.
\end{enumerate}
Given a set $M$ of real numbers, the supremum $\sup M$ of $M$ is defined as the sum set of $M$:
$$
\bigcup M=\bigcup_{\alpha\in M} \alpha.
$$
In general the set $M$ is a set of all elements $m$ of $\R$ which satisfies the given properties. In the above case, in the sense that the definition of $\sup M=\cup M\in \R$ starts from $\R$ and then defines an element $\sup M$ of $\R$, the definition is an impredicative definition.

One might think he seems to be able to refute the above criticism as follows. The above procedure just describes a process of choosing an element $\sup M$ from the set $\R$, but does not create the element $\sup M$ itself by the definition. However when writing the class of all sets by $C$, one can then say that the set $\{x\ |\ x\in C,\ x\not\in x\}$ is just choosing elements $x\in C$ such that $x\not\in x$.
Thus if the definition of $\sup M$ is allowed, then the Russell's set must be allowed to exist.

Those considerations would tell that just the exclusion of cyclic arguments might exclude other necessary and useful things, even if it can exclude paradoxes.

On the basis of such backgrounds, it has been understood that it is useful for the purpose of avoiding contradictions to prescribe sets by defining set theory as a formal axiomatic system. At this point it could be said that the usefulness of the concept of formal system has become to be understood. This is the fact that makes us realize again that the axiomatic description that goes back to Euclid is still effective and useful in the present age.

The Tarski's theorem stated in subsection \ref{10.4}:
\begin{quote}
The set of true sentences of a language ${\cal L}$ is not referred to by a sentence inside the language ${\cal L}$. Namely the predicate $T$ showing the truthness must not be inside the language ${\cal L}$.
\end{quote}
tells the similar thing. Namely summarizing, we have
\begin{quote}
The truthness of things is not the one which can be referred to directly to itself. It is recognized by systematic descriptions through axiomatic  formulations.
\end{quote}

\subsection{System as a restriction of self-reference}\label{12.4}

\normalsize

As we have seen, the cause of paradoxes or contradictions is self-reference, which goes back to the language and its symbolization that humans possess. We also have seen that it is equivalent to abandon being humans to remove this cause, however. In this situation we have seen that the way which we can adopt is the moderate course that, not by trying to remove the problems by concentrating on a local point like impredicative definition but by standing on a global viewpoint to the problems, we try to avoid contradictions by rewriting whole things as a formal system.

If we write down these with limiting to mathematics, it will be as follows.


\begin{thm}\label{th12.1}{\sl 
\begin{enumerate}
\item[ 1)]
If we allow complete self-reference, there will be a contradiction if we do not restrict the objects of thought to finite numbers.
\item[ 2)]
If it is possible to restrict self-reference completely, there will be no contradictions even if infinity exists. However in this case, we could not speak of ourselves as well as of the allowed infinity. Because speaking necessarily accompanies self-reference. Nevertheless it would be possible to avoid the current contradictions by making appropriate restrictions to the deed of speaking according to the order of necessity.
\end{enumerate}}
\end{thm}

In fact the axiomatic set theory always is at the risk of meeting contradictions. But at the point when a contradiction is discovered, it would be possible to retrieve the current consistency by making new restrictions to the system of set theory.

\subsection{Conclusion}\label{12.5}

\normalsize

We have been stating the important points related to the G\"odel's incompleteness theorem. The lesson we have learned would be that there is no end point for everything. When we think all the problems have been solved, a new problem has already arisen. Humans might be destined to meet the `next' always. This `next' problem is often caused by humans themselves, and this would suggest that the `next' problem is produced by `self-reference.' The problem of self-reference will continue forever in this way.

Not only science but also whatever humans do is the deeds and statements that arise from his or her `Working Hypotheses.' All is thought to be the attempt to verify hypotheses.

G\"odel's result may be taken as a negative answer to the working hypothesis of Leibniz: ``by writing down mathematics in terms of symbols, we can generate all mathematical theorems by a machine." However as we have seen there was made an implicit assumption here. We have stated that this assumption to identify the meta level and the object level is the characteristic that humans possess. We have referred to the fact that humans learned in the 20th century the method to deal with such problems based on a global viewpoint of using the systemic method of axiomatic formal set theory, not by coping with the problems by taking a local viewpoint of removing cyclic arguments such as impredicative definitions. 

Probably the method mentioned above will be replaced by others with regarding it as a former `working hypothesis' or will be restated in other forms in the future. All is an infinite series of verifications and rebuilding hypotheses.

If only the machines that humans produced remain in the future and if the truth is only what those machines tell and there is no human who checks it, there will be no existence which tells things like G\"odel's theorem. In such a world, there will then be no problems and those machines will produce `all mathematical theorems' at peace.

The learning that humans got in the 20th century to write down mathematics in axiomatic system without referring to the truthness directly might suggest the direction that humans would entrust themselves to such a `system' in the future.

Things like these are the problems that the younger generations are meeting actually in the present age, and those people will see the answers in the future. The author hopes that what has been stated in the present article would be able to tell those younger generations the course that humans traced and would be helpful as any suggestions to their future.




\end{document}